\documentclass[11pt, a4paper]{amsart}

\usepackage{amsthm,amsfonts,amsbsy,amssymb,amsmath,amscd}
\usepackage[all]{xy}
\usepackage[colorlinks=true]{hyperref}
\usepackage{pgf, tikz}

\newcommand{\QQ}{\mathbb{Q}}
\newcommand{\RR}{\mathbb{R}}

\newcommand{\PP}{\mathbb{P}}
\newcommand{\ZZ}{\mathbb{Z}}

\newcommand{\CE}{\mathcal{E}}

\newcommand{\CO}{\mathcal{O}}

\newcommand{\alb}{{\rm alb}}

\newcommand{\Pic}{{\mathrm{Pic}}}

\newcommand{\ch}{{\mathrm{ch}}}

\newcommand{\nt}{{\mathrm{nt}}}
\newcommand{\nti}{{\mathrm{int}}}

\newcommand{\vol}{{\mathrm{vol}}}

\newcommand{\rounddown}[1]{{\lfloor #1 \rfloor}}

\begin{document}
\title[Relative Clifford inequality]{Relative Clifford inequality for varieties fibered by curves}

\author{Tong Zhang}
\date{\today}

\address{Department of Mathematics, Shanghai Key Laborotary of PMMP, East China Normal University, 500 Dongchuan Road, Shanghai 200241, People's Republic of China}

\address{Department of Mathematical Sciences, Durham University, Durham, DH1 3LE, United Kingdom}

\email{tzhang@math.ecnu.edu.cn, mathtzhang@gmail.com}

\begin{abstract}
	We prove a sharp relative Clifford inequality for relatively special divisors on varieties fibered by curves. It generalizes the classical Clifford inequality about a single curve to a family of curves. It yields a geographical inequality for varieties of general type and Albanese-fibered by curves, extending the work of Horikawa, Persson, and Xiao in dimension two to arbitrary dimensions. We also apply it to deduce a slope inequality for some arbitrary dimensional families of curves. It sheds light on the existence of a most general Cornalba-Harris-Xiao type inequality for families of curves.
	
	The whole proof is built on a new tree-like filtration for nef divisors on varieties fibered by curves. One key ingredient of the proof is to estimate the sum of all admissible products of nef thresholds with respect to this filtration.
\end{abstract}

\maketitle

\setcounter{tocdepth}{1}
\tableofcontents


\theoremstyle{plain}
\newtheorem{theorem}{Theorem}[section]
\newtheorem{lemma}[theorem]{Lemma}
\newtheorem{coro}[theorem]{Corollary}
\newtheorem{prop}[theorem]{Proposition}
\newtheorem{defi}{Definition}[section]
\newtheorem{conj}[theorem]{Conjecture}
\newtheorem{ques}[theorem]{Question}

\newtheorem*{conj*}{Conjecture}
\newtheorem*{ques*}{Question}

\theoremstyle{remark}
\newtheorem{remark}{Remark}[section]
\newtheorem{assumption}[theorem]{Assumption}
\newtheorem{example}[theorem]{Example}

\numberwithin{equation}{section}


\section{Introduction}
The classical Clifford inequality states that for a special divisor $L$ on a smooth projective curve $X$ of genus $g \ge 2$, we have
\begin{equation} \label{classicalClifford}
	h^0(X, L) \le \frac{1}{2} \deg L + 1.
\end{equation}
Starting from this inequality is the rich theory of special divisors on algebraic curves. We refer to the reader to \cite[Chapter III]{ACGH_Geometry_curves1} for comprehensive details regarding this theory. 

The goal of this paper is to establish a relative (or family) version of \eqref{classicalClifford}. Suppose that $f: X \to Y$ is a fibration by curves of genus $g \ge 2$ between two varieties $X$ and $Y$, i.e., $X$ is a \emph{relative curve} over $Y$. In this paper, we prove a version of the Clifford inequality for numerically $f$-special divisors on $X$ (see Definition \ref{specialdivisor}). As a consequence, if $Y$ is of maximal Albanese dimension, then for a nef and numerically $f$-special divisor $L$ on $X$, we prove a sharp relative Clifford inequality between the $f$-continuous rank $h^0_f(X, \CO_X(L))$ of $L$ (see Definition \ref{continuousrank}) and the volume of $L$. As applications, we obtain
\begin{itemize}
	\item [(i)] a geographical inequality for varieties Albanese-fibered by curves;
	\item [(ii)] a slope inequality for (semi-stable) families of curves over abelian varieties, toric Fano varieties, and varieties covered by one of them.
\end{itemize}

Throughout this paper, we work over an algebraically closed field of characteristic zero. All varieties in this paper are assumed to be projective.

\subsection{Geographical slope $K^n / \chi$: the initial motivation}
Before introducing the relative Clifford inequality, we would like to introduce one of its applications to the geographical problem first. In fact, it is this application that motivates us to seek a general Clifford inequality.

The geography of varieties (usually of general type) is an important area in algebraic geometry. One central problem in this area concerns the distribution of birational invariants of varieties, with the purpose of applying it to characterize the geometry of varieties themselves.

A typical example reflecting this philosophy is the following conjecture of Reid \cite{Reid_Quadrics} concerning the geography of surfaces of general type. 
\begin{conj*} [Reid]
	For $g = 2$, $3$, $\cdots$, there exist rational numbers $a_g$ and $b_g$ with
	$$
	a_2 < a_3 < \cdots \quad \mbox{and} \quad  \lim\limits_{g \to \infty} a_g = 4
	$$
	such that for every (smooth) surface $X$ of general type, if $K_X^2 \le a_g \chi(X, \CO_X) - b_g$, then $X$ has a pencil of curves of genus at most $g$. 
\end{conj*}
The importance of this conjecture not only comes from the geography, it also implies several structural results about fundamental groups of certain surfaces (see \cite{Reid_Quadrics} for details). 

Reid's conjecture can be naturally generalized to higher dimensions. Without knowing the exact limit of $a_g$ a priori, at least it is natural to ask:

\medskip
\emph{For any $n > 2$, does a geographical inequality $K_X^n \le a_g \chi(X, \omega_X) - b_g$ for an $n$-dimensional variety $X$ of general type imply the existence of a covering family of curves on $X$ with the genus at most $g$?}
\medskip

Notice that questions of this type but between $K_X^n$ and $p_g(X)$ have already been considered recently by J. Chen and Lai in \cite[\S 6]{Chen_Lai_Small} for small $g$. 

In general, understanding the bound of the geographical slope $K^n / \chi$ is a very fundamental problem for algebraic varieties. The study of $K^2 / \chi$ for surfaces of general type dates back to the work of Italian school. Regarding this problem for irregular varieties, Pardini \cite{Pardini_Severi} proved the classical Severi inequality that $K_X^2 \ge 4\chi(X, \omega_X)$ for a minimal surface $X$ of general type and of maximal Albanese dimension. Recently, a generalized Severi inequality was obtained by Barja \cite{Barja_Severi} and independently by the author \cite{Zhang_Severi}, which states that $K_X^n \ge 2n! \chi(X, \omega_X)$ for an $n$-dimensional minimal variety $X$ of general type and of maximal Albanese dimension. It answers a question raised by Mendes Lopes and Pardini \cite[\S 5.2 (c)]{Lopes_Pardini_Geography}. In their very recent paper \cite{Barja_Pardini_Stoppino}, Barja, Pardini and Stoppino gave more inequalities of this type for varieties of maximal Albanese dimension under various assumptions. 

However, for most of the irregular varieties, the Albanese map is actually \emph{of fiber type}  (i.e., non-maximal Albanese dimension). From the perspective of Reid's conjecture, the lower bound of $K^n / \chi$  for these varieties should be an increasing function of the geometric genus (or the canonical volume) of the Albanese fiber, but such a desired lower bound seems far from being known for any $n>2$.\footnote{See \cite{Lopes_Pardini_Geography} for a survey of results for $n=2$ as well \cite{Roulleau_Urzua_Chern,Vidussi_Slope} for some recent results.}

Our first result in this paper addresses the above questions. Given any variety $X$, let $\alb_X: X \to A$ be the Albanese map of $X$. By the Stein factorization, the map $\alb_X$ factors through a fibration $f: X \to Y$ with $Y$ normal. We call $f$ the \emph{fibration induced by the Albanese map of $X$}. A fiber of $f$ is called an \emph{Albanese fiber}.

\begin{theorem} \label{main2}
	Let $X$ be an $n$-dimensional minimal variety of general type.\footnote{Throughout this paper, we refer the reader to Section \ref{minimality} for the notion of minimality.} Suppose that the Albanese map of $X$ induces a fibration of curves of genus $g \ge 2$. Then
	$$
	K_X^n \ge 2 n! \left(\frac{g-1}{g+n-2}\right)  \chi(X, \omega_X).
	$$
	In particular, if $g \ge n$, we have 
	$$
	K_X^n \ge n! \chi(X, \omega_X).
	$$
\end{theorem}

The surface case (i.e., $n=2$) of Theorem \ref{main2} is already known. Horikawa \cite[Theorem 2.1]{Horikawa_V} and Persson \cite[Proposition 2.12]{Persson_Chern} first proved that
\begin{equation} \label{HPX}
	K_X^2 \ge 4 \left(\frac{g-1}{g}\right) \chi(X, \omega_X)
\end{equation}
for minimal surfaces with hyperelliptic Albanese fibrations of genus $g \ge 2$. For general fibrations, \eqref{HPX} is due the work of Xiao \cite[Theorem 2]{Xiao_Slope}. In particular, their results have verified Reid's conjecture for irregular surfaces (e.g., taking $a_g = \frac{4g}{g+1}$ and $b_g = 1$). On the other hand, although not satisfying the exact assumption, Theorem \ref{main2} still makes sense when $X$ is a curve of genus $g \ge 2$ (i.e., $n=1$), which says that $\deg K_X \ge 2 \chi(X, \omega_X)$.

A notable fact is that the geographical slope in Theorem \ref{main2} also has the same limiting behavior as in \eqref{HPX}: for a fixed integer $n \ge 2$, the coefficients in Theorem \ref{main2} form an increasing sequence indexed by $g$, and we also have
$$
\lim\limits_{g \to \infty} 2 n! \left(\frac{g-1}{g+n-2}\right)  = 2n!
$$
in which the limit $2n!$ appears in the generalized Severi inequality. This is more than a coincidence. In fact, such a phenomenon has appeared already in Reid's conjecture for surfaces where the limit is $4 = 2 \cdot 2!$.\footnote{A similar asymptoticity for relative canonical divisors (namely $K^2_{X/\PP^1} / \deg (f_* \omega_{X/\PP^1})$) is very crucial in Pardini's proof of the classical Severi inequality in \cite{Pardini_Severi}.} 
From this point of view, Theorem \ref{main2} suggests that it is likely that a correct limit of $a_g$ in the aforementioned Reid's type question is $2n!$ for any $n \ge 2$, and Theorem \ref{main2} has confirmed this when the variety carries an Albanese pencil of curves. 

Previously, few results in this direction were known. Under the same assumption but further assuming $X$ is Gorenstein, Barja \cite[Theorem 4.1 and Remark 4.5]{Barja_Severi} proved that $K_X^n \ge 2(n-1)! \chi(X, \omega_X)$.\footnote{The author was informed by Barja that his result holds also in the $\QQ$-Gorenstein case.} His result is implied by Theorem \ref{main2}, because $2 n! \left(\frac{g-1}{g+n-2}\right) \ge 2(n-1)!$ when $g \ge 2$. In fact, Theorem \ref{main2} is sharper once $g \ne 2$. 

\begin{remark}
	Comparing Theorem \ref{main2} with \eqref{HPX}, we may wonder whether the better inequality $K_X^n \ge 2n! \left(\frac{g-1}{g}\right) \chi(X, \omega_X)$ holds for any $n \ge 2$. Unfortunately, based on the communication with us by Hu \cite{Hu2016}, we construct examples in Section \ref{examples} indicating that the above expectation is not true for any $n > 2$.
\end{remark}

\begin{remark}
	For a minimal surface $X$ of general type and Albanese-fibered by curves, we always have $K_X^2 \ge 2 \chi(X, \omega_X)$. Earlier than \eqref{HPX}, this weaker result was already stated and proved by Bombieri \cite[Lemma 14]{Bombieri_Canonical} in 1970s. The second inequality in Theorem \ref{main2} can be viewed as a higher dimensional analogue of this result, but we have to put a lower bound on the genus $g$. Otherwise the result would fail. In Section \ref{examples}, we provide examples in any dimension $n > 2$ with $K_X^2 < n! \chi(X, \omega_X)$. In those examples, the genus $g$ of Albanese fibers is at most $\frac{n+1}{2}$.
\end{remark}

\begin{remark}
	By the Miyaoka-Yau inequality, we know that $K^2_X \le 9 \chi(X, \omega_X)$ for complex surfaces of general type. We may ask whether there exists an upper bound of $K^n_X / \chi(X, \omega_X)$ for $n$-dimensional varieties $X$ ($n>2$) that are Albanese-fibered by curves. In Section \ref{unbounded}, we will show that $K^n_X / \chi(X, \omega_X)$ is unbounded from above in general, even when we assume that $\chi(X, \omega_X) > 0$. Notice that it does occur that $\chi(X, \omega_X) \le 0$ for varieties $X$ of dimension $n>2$ that are Albanese-fibered by curves. 
\end{remark}

As a direct consequence, Theorem \ref{main2} helps complete a detailed picture of the geography of irregular $3$-folds of general type. Similar to surfaces, the geography of $3$-folds of general type is also a vast and important subject. Some fundamental problems have been studied. For example, the Noether and Noether type inequalities in the Gorenstein case have been proved by J. Chen and M. Chen in \cite{Chen_Chen_Noether} and Hu in \cite{Hu2013}, respectively. The geography for $3$-folds fibered over curves were also studied by Ohno \cite{Ohno_Slope}, Barja \cite{Barja_Lower_bounds} and many others in the literature.

The next theorem, lying in this area, gives an explicit description on the geographical slope $K^3 / \chi$ of irregular $3$-folds of general type with respect to the Albanese dimension and the genus, or the volume, of the Albanese fiber. In particular, it fills in the blank for irregular $3$-folds of general type that are Albanese-fibered by curves, about which little was known before.

\begin{theorem} \label{main4}
	Let $X$ be an irregular and minimal $3$-fold of general type.
	\begin{itemize}
		\item [(1)] If $X$ is of maximal Albanese dimension, then
		$$
		K_X^3 \ge 12 \chi(X, \omega_X).
		$$
		\item [(2)] If $X$ has Albanese dimension two with general Albanese fiber $F$ a smooth curve of genus $g \ge 2$, then
		$$
		K_X^3 \ge 12 \left(\frac{g-1}{g+1}\right) \chi(X, \omega_X) = 12
		\left(\frac{\vol(K_F)}{\vol(K_F) + 4}\right) \chi(X, \omega_X).
		$$
		\item [(3)] If $X$ has Albanese dimension one with general Albanese fiber $F$ a smooth surface such that $(p_g(F), K^2_F) \ne (2, 1)$, then
		$$
		K_X^3 \ge 4 \left(\frac{K_F^2}{K_F^2+4}\right) \chi(X, \omega_X) = 4 \left(\frac{\vol(K_F)}{\vol(K_F) + 4}\right) \chi(X, \omega_X).
		$$
		Moreover, if $X$ is Gorenstein, then the inequality in (3) also holds when $(p_g(F), K^2_F) = (2, 1)$.
	\end{itemize}
\end{theorem}

Theorem \ref{main2} contributes to (2) here. (1) is simply from the generalized Severi inequality (the $n=3$ case of \cite[Corollary B]{Barja_Severi} or \cite[Theorem 1.1]{Zhang_Severi}). (3) is directly from \cite[Proposition 7.5, 7.6]{Zhang_3fold}, as well as \cite[Proposition 3.6]{Hu2013} for the Gorenstein case. 

It is natural to ask whether there is, in any dimension $n>3$, a similar result as Theorem \ref{main4} about $K^n/\chi$ for $n$-dimensional irregular varieties (of general type) of any given Albanese dimension that depends on the volume (or other birational invariants) of the corresponding Albanese fibers. We refer the reader to Question \ref{finalquestion} for an explicit description of our expectation.

\subsection{Relative Clifford inequality}
It turns out that Theorem \ref{main2} is just a special case of a more general result of this type for relatively special divisors which we will call the \emph{relative Clifford inequality} in the following. Before stating it, we first introduce some definitions.
\begin{defi} \label{continuousrank}
	Let $f: X \to Y$ be a morphism between two normal varieties $X$ and $Y$. For any divisor $L$ on $X$, we define the \emph{$f$-continuous rank of $L$} to be
	$$
	h^0_f (X, \CO_X(L)) := \min\{ \dim H^0 (X, \CO_X(L) \otimes f^*\alpha) | \alpha \in \Pic^0(Y) \}.
	$$
\end{defi}

The definition here is a relative version of the usual continuous linear series known as the form  $H^0(X, \CO_X(L) \otimes \alpha) \, (\alpha \in \Pic^0(X))$. The study of this continuous linear series dates back to the work of Mumford \cite{Mumford_Equation} and Kempf \cite{Kempf_Complex_Abelian_Varieties} on abelian varieties. For general irregular varieties, various properties of the continuous linear series have been intensively studied in the work of Pareschi and Popa \cite{Pareschi_Popa_Regularity,Pareschi_Popa_Regularity2}, Mendes Lopes, Pardini and Pirola \cite{Lopes_Pardini_Pirola_Continuous_Families_Divisors,Lopes_Pardini_Pirola_Brill_Noether}, and Barja \cite{Barja_Severi} during the past decades. Just to name a few. 

In fact, the notion here is compatible with all known ones. For any normal variety $X$, we may let $f$ be the structure morphism of $X$ (i.e., the morphism from $X$ to the spectrum of the base field). For this particular $f$, we simply have $h^0_f(X, \CO_X(L)) = h^0(X, \CO_X(L))$. Moreover, if we denote by $\alb_Y$ the Albanese map of $Y$ and let $a = \alb_Y \circ f$, then the definition here coincides with the continuous rank introduced by Barja in \cite{Barja_Severi}.\footnote{We did not use $\alb_Y$ literally in Definition \ref{continuousrank} because we would like to include fewer notations. Nevertheless, $\alb_Y$ is intrinsically involved in it.}

In the following, let $f: X \to Y$ be a fibration between two normal varieties $X$ and $Y$ with general fiber $C$ a smooth curve of genus $g \ge 2$.

\begin{defi} \label{specialdivisor}
	Let $L$ be a $\QQ$-Cartier Weil divisor on $X$. It is called \emph{$f$-special}, if $h^0(C, L|_C) > 0$ and $h^1(C, L|_C) > 0$. It is called \emph{numerically $f$-special}, if both $L|_C$ and $K_C-L|_C$ are pseudo-effective.
\end{defi}

The $f$-special divisor is a natural generalization of the classical special divisor. When $\dim X = 1$ and $f$ is the structure morphism of $X$ (i.e., $Y$ is a point), $f$-special just means special. 

It is clear that an $f$-special divisor is also numerical $f$-special. Notice that the assumption on $L$ in Definition \ref{specialdivisor} guarantees that $\deg(L|_C)$ is an integer, because a general $C$ lies in the smooth locus of $X$.

Now we state the relative Clifford inequality.

\begin{theorem} \label{main}
	Let $f: X \to Y$ be a fibration between two normal varieties $X$ and $Y$ with general fiber $C$ a smooth curve of genus $g \ge 2$. Assume that $\dim X = n \ge 2$ and that $Y$ is of maximal Albanese dimension. Let $L$ be a nef and numerically $f$-special divisor on $X$ with $d = \deg(L|_C) > 0$. Then
	$$
	h^0_f(X, \CO_X(L)) \le \left(\frac{1}{2n!} + \frac{n-\varepsilon}{n! d}\right) L^n.
	$$
	Here $\varepsilon = \frac{1}{2}$ when $d = 1$ or when $C$ is hyperelliptic, $L$ is $f$-special and $d$ is odd. Otherwise, $\varepsilon = 1$.	
\end{theorem}

When $d = \deg L > 0$, the classical Clifford inequality \eqref{classicalClifford} for special divisors $L$ can be reinterpreted as
$$
h^0_f( X, \CO_X(L) ) \le \left(\frac{1}{2} + \frac{1}{d}\right) \deg L.
$$
Here $f$ denotes the structure morphism of $X$, which is exactly of relative dimension one. This is the reason why we call Theorem \ref{main} a \emph{relative} Clifford inequality as it concerns a \emph{relative} curve as well as \emph{relatively} special divisors on it. In fact, these two inequalities above are of the same type. We could just literally take $n=1$ and $\varepsilon = 0$ in Theorem \ref{main} to get \eqref{classicalClifford}. Moreover, Theorem \ref{main} for $L=K_X$ implies Theorem \ref{main2}. 

Before going further, we want to illustrate the necessity of all assumptions in Theorem \ref{main} by the following remarks, because these assumptions reflect the significant differences when dealing with a relative curve rather than an absolute curve. Moreover, we emphasize here that the assumptions are probably the \emph{weakest} ones in order to guarantee the inequality in Theorem \ref{main}.

\begin{remark} \label{d>0} 
	The assumption that $d > 0$ is indispensable for obtaining Theorem \ref{main} (not merely because $d$ appears on the denominator). Actually, consider the trivial fibration $f: A \times C \to A$, where $A$ is an abelian variety of dimension $n-1$, and $C$ is a smooth curve of genus $g \ge 2$. If we choose $L = f^*H$ for an ample divisor $H$ on $A$, then $L$ is nef, $f$-special and $L^n = 0$. However, it is fairly easy to see that 
	$$
	h^0_f(X, \CO_X(L)) = h^0(A, \CO_A(H)) = \frac{H^{n-1}}{(n-1)!} > 0.
	$$
	In other words, when $d = 0$, unlike the classical Clifford inequality \eqref{classicalClifford}, there is no way to bound $h^0_f(X, \CO_X(L))$ from above just by $L^n$ and a certain absolute constant independent on $L$. 
\end{remark}

\begin{remark} \label{epsilon=1}
	It is impossible to take $\varepsilon = 1$ in all cases, especially when $C$ is hyperelliptic and $d$ is odd. In Section \ref{epsilon}, for any odd integer $1 \le d < 2g-2$, we construct a hyperelliptic surface fibration $f: X \to A$ and a nef and $f$-special divisor $L$ on $X$ with $d = \deg(L|_C)$ and 
	$$
	h^0_f(X, \CO_X(L)) > \left(\frac{1}{4} + \frac{1}{2d} \right) L^2.
	$$
	We also construct examples in any dimension $n \ge 2$ suggesting that we have to allow $\varepsilon = \frac{1}{2}$ at least when $d=1$. In other words, Theorem \ref{main} is sharp for $d=1$ in any dimension $n \ge 2$.
\end{remark}

\begin{remark} \label{a-bigness} 
	We may wonder whether the nefness assumption on $L$ in Theorem \ref{main} could be removed (while replacing $L^n$ by $\vol(L)$). The answer is \emph{no}. In Section \ref{nefnessremark}, for any $n \ge 2$ we construct a fibration $f: X \to A$ from a smooth variety $X$ of dimension $n$ to an abelian variety $A$ of dimension $n-1$ with general fiber $C$ a smooth curve of genus $g \ge 2$ such that for any $2 \le d \le 2g-2$, there exists an $f$-special and \emph{non-nef} divisor $L$ on $X$ with $\deg(L|_C) = d$ but
	$$
	h^0_f(X, \CO_X(L)) > \left(\frac{1}{2n!} + \frac{n - \varepsilon}{n! d}\right) \vol(L).
	$$
	
	In contrast to this remark is a very recent paper \cite{Barja_Pardini_Stoppino} by Barja, Pardini and Stoppino about the Clifford-Severi type inequality for line bundles on varieties of maximal Albanese dimension, in which the nefness assumption on the line bundle $L$ is removed without weakening the result in their previous work (such as \cite{Barja_Severi}). As we have seen, we cannot expect the similar phenomenon any more as long as the variety is \emph{not} of maximal Albanese dimension. This phenomenon somehow indicates that dealing with fiber type Albanese maps is more subtle than generically finite ones.
\end{remark}

From Remark \ref{epsilon=1}, we know that Theorem \ref{main} for $d=1$ is already sharp. Actually, we also consider the general case when $d > 1$ and prove the following:

\begin{coro} \label{2n-1!}
	Under the setting in Theorem \ref{main}, let $d > 1$. Then
	\begin{itemize}
		\item [(i)] we have
		$$
		L^n \ge 2(n-1)! h^0_f(X, \CO_X(L)).
		$$
		
		\item[(ii)] The above inequality is sharp. Suppose that the equality holds for $L$ and that $h^0_f(X, \CO_X(L)) > 0$ (or equivalently, $L$ is big). Then $d=2$ and $Y$ is birational to an abelian variety of dimension $n-1$.
	\end{itemize}
\end{coro}

Under a stronger assumption that $L$ is nef, subcanonical (i.e., $K_X-L$ is pseudo-effective) with its continuous moving part being $f$-big (which we will mention later), Corollary \ref{2n-1!} (i) was proved by Barja in \cite[Main Theorem (b)]{Barja_Severi}. Here the inequality is just a simple consequence of Theorem \ref{main}, and it holds in a much broader setting than \cite{Barja_Severi}. More than the inequality itself, in Section \ref{epsilon}, we construct examples of $(X, L)$ in any dimension for which this inequality becomes an equality. Moreover, we also give a characterization of the equality case. 

\begin{remark}
	It is well-known that if the equality in \eqref{classicalClifford} holds (so $d$ is even) and $\CO_X(L) \ncong \CO_X$ or $\omega_X$, then $X$ must be a hyperelliptic curve. However, this is no longer true under the relative setting. A simple example is a surface fibration $f: X \to E$ from a smooth surface $X$ to an elliptic curve $E$ with a non-hyperelliptic general fiber $C$  as well as a section $D$ with $D^2 = -e < 0$. Such an example can be constructed either by blowing up the product $C \times E$ or by using a high degree ($>2$) cover over a ruled surface over $E$. Consider the divisor $L=2D + 2eC$ on $X$. It is easy to check that $L$ is nef and $f$-special, $\deg(L|_C) = 2$, $L^2 = 4e$, and $h^0_f(X, \CO_X(L)) = h^0_f(X, \CO_X(2eC)) = 2e$. 
	
	Notice that if we use the blowing up of $C \times E$ in the above, then $K_X$ is not big. Hence $K_X-L$ is not pseudo-effective. Moreover, the continuous moving part is not $f$-big. So it violates the assumption in \cite[Main Theorem (b)]{Barja_Severi}. Nevertheless, Theorem \ref{main} can apply.
\end{remark}

\subsection{Slope inequality for families of curves} In this subsection, we introduce another application of the relative Clifford inequality to establishing a Cornalba-Harris-Xiao type inequality for higher dimensional families of curves.

An important invariant in the study of families of curves is the \emph{slope}. For a relatively minimal (non-isotrivial) surface fibration $f: X \to Y$ of curves of genus $g \ge 2$, the slope is defined to be 
$$
s(f) : = \frac{K_{X/Y}^2}{\deg (f_* \omega_{X/Y})}.
$$
A fundamental result about the slope is the so-called \emph{slope inequality}, or \emph{Cornalba-Harris-Xiao inequality}, proved by Cornalba and Harris \cite[Theorem 1.3]{Cornalba_Harris_Slope} and Xiao  \cite[Theorem 2]{Xiao_Slope} which asserts that $s(f) \ge \frac{4g-4}{g}$, i.e.,
$$
K_{X/Y}^2 \ge 4 \left( \frac{g-1}{g} \right) \deg (f_* \omega_{X/Y}).
$$
We refer the reader to \cite[Chapter XIV]{ACG_Geometry_curves2} as well as all references therein for details regarding this inequality. In \cite{Zhang_Slope}, using a characteristic $p>0$ method, the author managed to prove a slope inequality for relatively minimal 3-fold fibrations $f: X \to Y$ of curves of genus $g \ge 2$ which asserts that
$$
K_{X/Y}^3 \ge 12 \left( \frac{g-1}{g+1} \right)  \ch_2(f_* \omega_{X/Y}).
$$
Moreover, a question \cite[Question in \S 1]{Zhang_Slope} was proposed: \emph{for $m \ge 2$, does there exist a general Cornalba-Harris-Xiao inequality for $m$-dimensional families of curves $f: X \to Y$ 
between $K_{X/Y}^{m+1}$ and $\ch_{m}(f_* \omega_{X/Y})$?}\footnote{There is another generalization of the slope inequality with one dimensional bases and high dimensional fibers: inspired by \cite{Cornalba_Harris_Slope}, Barja and Stoppino made a conjecture \cite[Page 36, Conjecture  1]{Barja_Stoppino} on a slope inequality of similar type for families of higher dimensional varieties over curves, which we will address in a forthcoming paper \cite{Hu_Zhang}.}

As another main result of this paper, we give the first evidence to the existence of such a general inequality for \emph{arbitrary dimensional} families of curves of genus $g \ge 2$.

\begin{theorem} \label{main3}
	Let $f: X \to Y$ be a semi-stable fibration of curves of genus $g \ge 2$ from a normal variety $X$ of dimension $n \ge 2$ to a smooth variety $Y$ of dimension $n-1$. Suppose that $K_{X/Y}$ is nef. Then the inequality
	$$
	K_{X/Y}^n \ge 2 n! \left(\frac{g-1}{g+n-2}\right) \ch_{n-1}(f_* \omega_{X/Y})
	$$
	holds if there is a finite morphism $Z \to Y$ where $Z$ is either an abelian variety or a smooth toric Fano variety.
	
    Moreover, if $Y$ itself is an abelian variety, then the semi-stability assumption on $f$ can be removed.
\end{theorem}
Here by \emph{semi-stability} we mean that each fiber of $f$ is semi-stable in the sense of Deligne and Mumford. 

We may view abelian varieties and toric Fano varieties as higher dimensional generalizations of elliptic curves and $\PP^1$ in the dynamical sense. There are some classification results about what the varieties $Y$ and $Z$ in Theorem \ref{main3} can possibly be. Notably, $Y$ does not have to be of maximal Albanese dimension. For example, $Y$ can be $\PP^{n-1}$, certain blow-up of $\PP^{n-1}$, the product of $\PP^{m}$ and any abelian variety of dimension $n-m-1$, etc. We refer the reader to \cite{Debarre_Abelian_Cover,Hwang_Mok_Abelian_Cover,Demailly_Hwang_Peternell_Abelian_Cover} for details about varieties that can be covered by abelian varieties, and to \cite{Batyrev_Toric3,Batyrev_Toric4,Kreuzer_Nill_Toric5} about toric Fano varieties in lower dimensions.  

Another notable fact is that in the proof of Theorem \ref{main3}, we actually discover that $\ch_{n-1}(f_*\omega_{X/Y}) \ge 0$. It can be viewed as an analogue of the semi-positivity result by Fujita \cite{Fujita_Fibration} that $\deg (f_* \omega_{X/Y}) \ge 0$ when $Y$ is a curve.

\subsection{Idea the proof} In this subsection, we sketch the proofs of the above theorems.  The whole proof is based on a new \emph{tree-like filtration} for nef divisors which is the main innovative technique of this paper. Roughly speaking, it is via this filtration that we deduce an explicit Clifford-type inequality, i.e., Theorem \ref{explicitclifford}, for nef divisors on varieties fibered by curves over \emph{arbitrary base varieties}. Then each of the main results is a limit version of Theorem \ref{explicitclifford} in the corresponding setting.

\subsubsection{Tree-like filtration} Let $f: X \to Y$ be a fibration of curves between two normal varieties $X$ and $Y$, where $\dim X = n \ge 2$. Pick a base-point-free linear system $|H|$ on $Y$ with $H$ big. Let $L$ be a nef divisor on $X$. Replacing $X$ by an appropriate blowing up induced by a pencil in $|H|$, we may assume that $X$ is in the meantime fibered over a curve with a general fiber $X_1$. Then by Theorem \ref{filtration}, we obtain a filtration of nef divisors
$$
L = L_0 > L_1 > \cdots > L_N \ge 0
$$
on $X$. Notice that $X_1$ is vertical with respect to $f$. Hence $X_1$ is also fibered by curves. By Theorem \ref{filtration}, we can get a similar filtration for each $L_i|_{X_1}$ by modifying $X_1$ accordingly. Keep doing this until we reach a fibered surface $X_{n-2}$. The whole process gives rise to a number of varieties $(X =) X_0$, $X_1$, $\cdots$, $X_{n-2}$ as well as a tree-like filtration for $L$ associated to these varieties as follows:
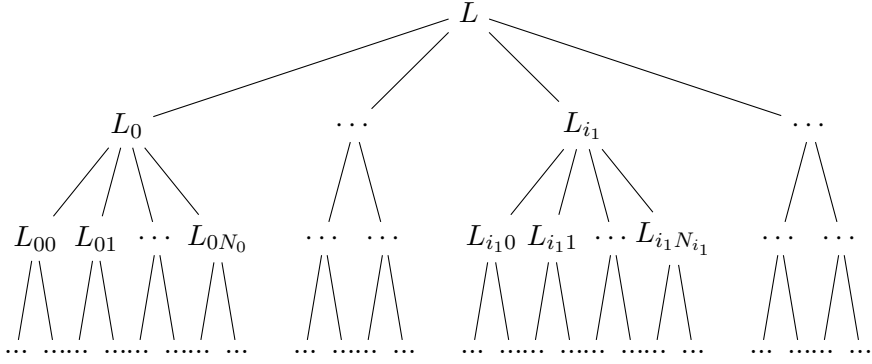
\begin{center}
	\begin{tikzpicture}
	[level 1/.style={sibling distance=30mm},
	level 2/.style={sibling distance=8mm},
	level 3/.style={sibling distance=5mm}]
	\node {$L$}
	child {node {$L_0$}
		child {node {$L_{00}$}
			child {node {...}}
			child {node {...}}
		}
		child {node {$L_{01}$}
			child {node {...}}
			child {node {...}}
		}
		child {node {$\cdots$}
			child {node {...}}
			child {node {...}}
		}
		child {node {$L_{0N_0}$}
			child {node {...}}
			child {node {...}}
		}
	}
	child {node {$\cdots$}
		child {node {$\cdots$}
			child {node {...}}
			child {node {...}}
		}
		child {node {$\cdots$}
			child {node {...}}
			child {node {...}}
		}
	}
	child {node {$L_{i_1}$}
		child {node {$L_{i_10}$}
			child {node {...}}
			child {node {...}}
		}
		child {node {$L_{i_11}$}
			child {node {...}}
			child {node {...}}
		}
		child {node {$\cdots$}
			child {node {...}}
			child {node {...}}
		}
		child {node {$L_{i_1N_{i_1}}$}
			child {node {...}}
			child {node {...}}
		}
	}
	child {node {$\cdots$}
		child {node {$\cdots$}
			child {node {...}}
			child {node {...}}
		}
		child {node {$\cdots$}
			child {node {...}}
			child {node {...}}
		}
	};
	\end{tikzpicture}
\end{center}
Here each $L_{i_1 \cdots i_m}$ $(1 \le m \le n-1)$ that appears as a node in the above tree is a nef divisor on $X_{m-1}$. We refer the reader to Section \ref{iteratedfiltration} for more details. 

\subsubsection{Key estimate} In this paper, we use the above tree-like filtration to study the relation between $h^0(X, \CO_X(L))$ and $L^n$. In fact, for a nef divisor $L$ on $X$, from the above tree-like filtration we can get the following number
\begin{equation} \label{keynumber}
	\sum_{i_1, i_2, \cdots, i_{n-1}} a_{i_1} a_{i_1 i_2} \cdots a_{i_1 i_2 \cdots i_{n-1}}
\end{equation}
associated to $L$ (and $H$). Here for each multi-index $i_1 \cdots i_m$ of length $m$, 
$$
a_{i_1 \cdots i_m} := \nti_{\pi_{m-1}}(L_{i_1 \cdots i_m})
$$
is the smallest integer that is bigger than the nef threshold of $L_{i_1 \cdots i_m}$ with respect to $\pi_{m-1}$, where $\pi_{m-1}$ is the fibration from $X_{m-1}$ to a curve (see Section \ref{nefthreshold} for the precise definition). The number in \eqref{keynumber} is the whole key to us, because it offers a crucial link between $h^0(X, \CO_X(L))$ and $L^n$. In particular, bounding \eqref{keynumber} in different directions leads to the following explicit estimate:

\begin{theorem} [Clifford inequality over an arbitrary base] \label{explicitclifford}
	Let $f: X \to Y$ be a fibration between two normal varieties $X$ and $Y$ with $\dim X = n \ge 2$ and with general fiber $C$ a smooth curve of genus $g \ge 2$. Fix a base-point-free linear system $|H|$ on $Y$ with $H^{n-1} > 0$. Let $L$ be a nef and numerically $f$-special divisor on $X$ with $d = \deg(L|_C) > 0$. Then one of the following results holds:
	\begin{itemize}
		\item [(1)] We have
		$$
		h^0(X, \CO_X(L)) \le \left(\frac{1}{2n!} + \frac{n-\varepsilon}{n! d}\right) L^n + \frac{d+2}{2} \sum_{m=1}^{n-1} B_m(X, H, L).
		$$
		Here $\varepsilon = \frac{1}{2}$ when $d = 1$ or when $C$ is hyperelliptic, $L$ is $f$-special and $d$ is odd. Otherwise, $\varepsilon = 1$. 
		\item [(2)] We have
		$$
		h^0(X, \CO_X(L)) \le h^0(X', \CO_{X'}(L)),
		$$ 
		where $X' = f^*H$ is viewed as a subvariety of $X$.
	\end{itemize}
\end{theorem}
We refer the reader to Definition \ref{bmp} in Section \ref{settingp} for the concrete expression of $B_m(X, H, L)$. In one word, this theorem says that either we can almost get the desired comparison between $h^0(X, \CO_X(L))$ and $L^n$ on $X$, or we can reduce the question to a variety of lower dimension.

\subsubsection{Proof of main results} As we have mentioned, all proofs of main results are heavily based on Theorem \ref{explicitclifford}. For the proof of Theorem \ref{main}, we adopt the idea of Pardini \cite{Pardini_Severi} to construct \'etale covers $X_{[k]}$ of $X$ via the multiplication-by-$k$ map of the Albanese variety of $Y$. We are able to show that Theorem \ref{main} is in fact a limit version of Theorem \ref{explicitclifford} on all $X_{[k]}$. By the generic vanishing theorem of Green-Lazarsfeld \cite{Green_Lazarsfeld_GenericVanishing}, we can show that
$$
h^0_f(X, \omega_X) \ge \chi(X, \omega_X)
$$
when the Albanese map of $X$ induces a fibration of curves. Hence Theorem \ref{main2} can be proved. 

To prove of the slope inequality in Theorem \ref{main3}, we may just reduce to the case when $Z=Y$ via a base change. Let us first assume that $Y$ itself is a toric Fano variety. Then we apply the nontrivial polarized endomorphism of $Y$ to construct a family of fibrations $f_{(k)}: X_{(k)} \to Y$ via base changes. With the help of Koll\'ar's vanishing theorem \cite{Kollar_Higher} on $H^i(Y, f_{(k)*} \omega_{X_{(k)}/Y})$, we show that Theorem \ref{main3} is another limit version of Theorem \ref{explicitclifford}. When $Y$ is an abelian variety, the result is more or less equivalent to Theorem \ref{main}.

\begin{remark}
	It is impossible to prove Theorem \ref{main3} directly using Theorem \ref{main} or some related methods about irregular varieties. For example, most toric Fano varieties cannot be covered by abelian varieties (except in dimension one), and this already prevents us from transferring the slope inequality over toric Fano varieties to that over abelian varieties via base changes. From this point of view, Theorem \ref{explicitclifford} is useful in its own right, because not only for irregular varieties, it applies also to some problems which actually involves no irregular varieties.
\end{remark}

\subsection{Comparison with other methods}
In \cite{Zhang_Severi}, the author studied the canonical volume of varieties of maximal Albanese dimension. The method therein may apply to the current setting, but it is too weak to give results like Theorem \ref{main2}, even in the Gorenstein $3$-fold case (see \cite{Zhang_3fold} for example, where the result is much weaker than here). 

In \cite{Barja_Severi}, Barja introduced an interesting method to study the lower bound of the volume of nef line bundles on irregular varieties. Let $a: X \to A$ be a morphism from a projective variety $X$ of dimension $n$ to an abelian variety $A$. For any nef line bundle $L$ on $X$, Barja considered the so-called continuous moving part $M$ of $L$, with properties that $L-M$ is effective and that $h^0_a(X, L) = h^0_a(X, M)$ \cite[Definition 3.1]{Barja_Severi}. When $\dim a(X) < \dim X$, he reduces the lower bound of $L^n$ (c.f. \cite[Main Theorem (b)]{Barja_Severi}) to a lower bound of $M^n$, by assuming the $a$-bigness of $M$ and the subcanonicity of $L$. 

However, it can happen that $M$ is not $a$-big even if $L$ itself is big, and Barja \cite[Remark 3.8]{Barja_Severi} has observed this when $\dim X - \dim a(X) \ge 2$. We discover that even if the relative dimension is one (which is our main interest in this paper), this still happens. See Remark \ref{abigness}. Second, in order to get \cite[Main Theorem (b)]{Barja_Severi}, the method of Barja requires a very strong assumption that $K_X-L$ is pseudo-effective. A priori, $K_X$ has to be big (i.e., $X$ is of general type), otherwise $L^n=h^0_a(X, L) = 0$ and the result loc. cit. becomes vacuous. While our method in the current paper only requires a relative assumption that $K_C-L|_C$ is pseudo-effective. This is a much weaker one.

Aside from the above, the most significant difference between our method in the current paper and the methods in \cite{Barja_Severi} or \cite{Zhang_3fold} lies in the following. Say $L=K_X$ for simplicity. When the Albanese dimension of $X$ is $k$, the key estimate used by Barja is as follows:
\begin{equation} \label{bestimate}
	K_X^n \ge 2k!h^0_a(X, \omega_X) + \mbox{error term}.
\end{equation}
Barja proved this inequality based on Xiao's method for fibrations over $\PP^1$ (see \cite[\S 5.1]{Barja_Severi}). When $n=3$, a similar estimate \cite[Theorem 5.2]{Zhang_3fold}, replacing $h^0_a(X, \omega_X)$ in \eqref{bestimate} by $h^0(X, \omega_X)$, is also crucial in \cite{Zhang_3fold}. In the case when $k=n$, this estimate is good enough to get a sharp bound after combining Pardini's limiting method \cite{Pardini_Severi}. However, Theorem \ref{explicitclifford} (together with the continuation of the current paper which is being written) reveals that when $k < n$, the main term is actually much larger than exhibited in \eqref{bestimate}. In particular, the leading coefficient increasingly tends to $2(k+1)!$. This very important feature has not been detected before via the methods in either \cite{Barja_Severi} or \cite{Zhang_3fold}, but the method via the tree-like filtration does capture it! This is the reason why results in this paper are sharper than the aforementioned ones. It is just due to this sharper result that we are able to give the characterization in Corollary \ref{2n-1!} (ii). 

We expect that the method introduced in this paper can play a similar role when studying fibrations over high dimensional varieties as what Xiao's method \cite{Xiao_Slope} does for fibrations over curves. Notice that Xiao's method has already been applied further to study families of higher dimensional varieties over curves (e.g., \cite{Ohno_Slope,Barja_Lower_bounds,Barja_Stoppino}), geographical problems (e.g., \cite{Pardini_Severi,Barja_Severi}), and even some conjectures of arithmetic background (e.g., \cite{Chen_Lu_Zuo_Oort}).

\subsection{Notation and Conventions} \label{notation} The following notation and conventions will be used throughout this paper.

\subsubsection{Minimality} \label{minimality} We say that a variety $X$ is minimal, if $X$ is normal with at worst terminal singularities, and $K_X$ is a nef $\QQ$-Cartier Weil divisor (i.e., $X$ being $\QQ$-Gorenstein). We say that a normal variety $X$ is of general type, if $K_X$ is big.

\subsubsection{Fibration and nef threshold} \label{nefthreshold} A fibration in this paper always means a surjective morphism with connected fibers. Let $f: X \to B$ be a fibration over a smooth curve $B$ with a general fiber $F$. For any nef divisor $L$ on $X$, the \emph{nef threshold of $L$ with respect to $f$} refers to the following real number
$$
\nt_f(L) := \sup \{ a \in \RR | L - aF \, \mbox{is nef}\}.
$$
In the above setting, we write
$$
\nti_f(L) := \rounddown{\nt_f(L)} + 1 = \min \{ a \in \ZZ | L - aF \, \mbox{is not nef}\}.
$$
This slightly bigger number is always integral and was introduced in \cite{YuanZhang_RelNoether,Zhang_Severi}.

\subsubsection{Divisor and line bundle} In this paper, we do not distinguish integral divisors and line bundles on smooth varieties. In particular, if $L$ is an integral divisor on a smooth variety $V$, then sometimes we denote $h^0(V, \CO_V(L))$ by $h^0(V, L)$. 

If $C$ is a smooth hyperelliptic curve, we use $g_2^1$ to denote an effective divisor $D$ on $C$ with $\deg(D)=2$ and $h^0(C, D) = 2$.

\subsubsection{Horizontal and vertical divisor} \label{horizontalvertical} Let $f: X \to Y$ be a fibration between two varieties $X$ and $Y$. Let $D$ be a prime divisor on $X$. We say that $D$ is \emph{horizontal (resp. vertical) with respect to $f$}, if $f(D) = Y$ (resp. $f(D) \subsetneq Y$). In general, a ($\QQ$-)divisor $D$ on $X$ is \emph{horizontal (resp. vertical) with respect to $f$} if $D$ is a ($\QQ$-)linear combination of prime divisors that are horizontal (resp. vertical) with respect to $f$. For any $\QQ$-divisor $D$ on $X$, we can uniquely write
$$
D = D^H + D^V,
$$
where $D^H$ (resp. $D^V$) is horizontal (resp. vertical) with respect to $f$. We call $D^H$ (resp. $D^V$) the horizontal (resp. vertical) part of $D$ with respect to $f$. 

\subsection{Structure of the paper} This paper is organized as follows. From Section \ref{generalresult} to Section \ref{moreinequality}, we focus on the construction of the tree-like filtration for nef divisors on varieties fibered by curves. Furthermore, we deduce a number of numerical inequalities which will serve for the proof of Theorem \ref{explicitclifford}. The proof of Theorem \ref{explicitclifford} will be presented in Section \ref{Clifforderror} and \ref{proofClifford}. Proofs of Theorem \ref{main}, \ref{main2} and \ref{main3} are presented in Section \ref{proofmain}. In Section \ref{examples}, we give several interesting examples. Those examples will explain more explicitly all remarks we have made before. Corollary \ref{2n-1!} will be proved in Section \ref{proofcoro}. Finally, we raise a question of Reid's type about irregular varieties of non-maximal Albanese dimension in Section \ref{question}.

\subsection*{Acknowledgement} The author would like to thank Christopher Hacon and Zhi Jiang for their interest in this paper. The author also would like to thank Miguel Barja for communications on some results in \cite{Barja_Severi,Barja_Pardini_Stoppino}, as well as Rita Pardini and Stefano Vidussi for their valuable comments on the first version of the paper. Special thanks go to Yong Hu for showing the interesting example in \cite{Hu2016} and for his generosity allowing us to apply his construction in this paper, and to Zhixian Zhu for inspiring discussions on toric varieties.

Part of this paper was written when the author was visiting Beijing International Center of Mathematical Research in 2017. The author would like to thank Chenyang Xu for the invitation and the institute for its hospitality.

The work was funded by a Leverhulme Trust Research Project Grant ECF-2016-269.

\section{Nef divisor on varieties fibered over curves} \label{generalresult}

In this section, we assume that $f: X \to B$ is a fibration from a smooth variety $X$ of dimension $n$ to a smooth curve $B$ with a general fiber $F$. Let $L \ge 0$ be a nef divisor on $X$. Recall that we have the following theorem which is just a slightly simpler reformulation of \cite[Theorem 2.3]{Zhang_Severi}.

\begin{theorem}  \cite[Theorem 4.1]{Zhang_Slope} \label{filtration}
	Let $f: X \to B$, $F$ and $L$ be as above. Then there is a birational morphism $\sigma: X_L \to X$ and a sequence of triples
	$$
	\{(L_i, Z_i, a_i), \quad i=0, 1, \cdots, N\}
	$$
	on $X_L$ with the following properties:
	\begin{itemize}
		\item $(L_0, Z_0, a_0)=(\sigma^*L, 0, \nti_{f_L}(L_0))$ where $f_L: X_L \stackrel{\sigma}{\to} X \stackrel{f}{\to} B$ is the induced fibration. 
		\item For any $i=0, \cdots, N-1$, there is a decomposition
		$$
		|L_i-a_iF_L|= |L_{i+1}| + Z_{i+1}
		$$
		such that $Z_i \ge 0$ is the fixed part of $|L_i-a_iF_i|$ and that the movable part $|L_{i+1}|$ of $|L_i-a_iF_i|$ is base point free. Here $F_L = \sigma^*F$ denotes general fiber of $f_L$, and $a_{i+1} = \nti_{f_L}(L_{i+1})$.
		\item We have $h^0(X_L, L_N-a_NF_N)=0$.
	\end{itemize}
\end{theorem}

Briefly speaking, we obtain from the above theorem a filtration
$$
\sigma^* L = L_0 > L_1 > \cdots > L_N \ge 0
$$
of nef divisors on a birational model $X_L$ of $X$. For simplicity, we still denote by $F$ a general fiber of $f_L: X_L \to B$ in the rest of this section.

\begin{prop} \label{h0} 
	We have the following inequality:
	$$
	h^0(X, L) \le \sum_{i=0}^{N} a_i h^0(F, L_i|_F).
	$$
\end{prop}

\begin{proof}
	See \cite[Proposition 2.6 (1)]{Zhang_Severi}.
\end{proof}

For any $0 \le i \le N$, write
$$
L'_i := L_i-(a_i-1)F.
$$
By the definition of the nef threshold in Section \ref{nefthreshold}, $L'_i$ is a nef divisor on $X_L$, and we have 
\begin{equation}
	L'_i = a_{i+1} F + Z_{i+1}+L'_{i+1}. \label{l'i}
\end{equation}
In the following, let $P$ be any nef $\QQ$-divisor on $X$ and denote $P_0 = \sigma^* P$.

\begin{prop} \label{numericalinequality}
	We have the following numerical inequalities:
	\begin{itemize}
		\item [(1)] $\displaystyle L^n \ge n\sum_{i=0}^{N} a_iL^{n-1}_i F +  \sum_{i=1}^{N} \left((n-1)a_iL^{n-2}_iF +  (L'_i)^{n-1} \right) Z_i - nL^{n-1}F$;
		
		\item [(2)] $\displaystyle PL^{n-1} \ge  (n-1) \sum_{i=0}^{N} a_i P_0L^{n-2}_i F - (n-1) PL^{n-2}F$.
	\end{itemize}
\end{prop}

\begin{proof}
	For any $0 \le i \le N-1$, by (\ref{l'i}), we have\footnote{Here and in the rest of the paper, we will always use the following inequality: if $A_1, \cdots, A_n$, $B_1, \cdots, B_n$ are nef divisors and $B_i-A_i$ is pseudo-effective for any $i$, then $B_1 \cdots B_n \ge A_1 \cdots A_n$.}
	\begin{eqnarray*}
		(L'_i)^n-(L'_{i+1})^n & = & (L'_i - L'_{i+1}) \sum_{j=0}^{n-1} (L'_i)^j (L'_{i+1})^{n-1-j} \\
		& = & (a_{i+1} F + Z_{i+1}) \sum_{j=0}^{n-1} (L'_i)^j (L'_{i+1})^{n-1-j} \\
		& \ge & a_{i+1} \left(L_{i+1}^{n-1} F + \sum_{j=1}^{n-1}L_i^j L_{i+1}^{n-1-j} F\right) + (L'_{i+1})^{n-1}Z_{i+1} \\
		& \ge & a_{i+1} \left(L_{i+1}^{n-1}F + (n-1) L_i L_{i+1}^{n-2}F \right) +  (L'_{i+1})^{n-1}Z_{i+1} \\
		& = & n a_{i+1}L_{i+1}^{n-1} F + (n-1) a_{i+1} L_{i+1}^{n-2} F Z_{i+1} + (L'_{i+1})^{n-1}Z_{i+1}.
	\end{eqnarray*}
	Sum over all above $i$, and notice that $(L'_N)^n \ge 0$ and
	$$
	L^n = L_0^n = (L'_0)^n + n(a_0-1)L_0^{n-1}F.
	$$
	Thus the proof of (1) is completed.
	
	The proof of (2) goes in a similar way. We sketch it here. In fact, we can deduce that
	$$
	P_0(L'_i)^{n-1} - P_0(L'_{i+1})^{n-1} \ge (n-1)a_{i+1}P_0L_{i+1}^{n-2}F
	$$
	for any $0 \le i \le N-1$ using a computation very similar to the above. On the other hand,
	$$
	PL^{n-1} - P_0(L'_0)^{n-1} \ge (a_0-1)(n-1) P_0L_0^{n-2}F.
	$$
	Thus the proof of (2) is completed.
\end{proof}

\section{Tree-like filtration for nef divisors on varieties fibered by curves} \label{iteratedfiltration}

In this section, our goal is to construct the tree-like filtration for nef divisors on varieties fibered by curves.

\subsection{Initial data} \label{setting} Let $f: X \to Y$ be a fibration from a smooth variety $X$ of dimension $n$ to a smooth variety $Y$ of dimension $n-1$ with general fiber $C$ being smooth. We fix a smooth divisor $H$ on $Y$ such that 
\begin{itemize}
	\item $|H|$ is base point free;
	\item $H$ is also big, i.e., $H^{n-1} > 0$.
\end{itemize}
This is guaranteed by Bertini's theorem. Throughout this section, we assume that $L$ is an effective and nef divisor on $X$.

\subsection{First step} \label{firststep} Choose a very general pencil in $|H|$. Let $\tilde{Y}_0$ be the blowing up of the indeterminacies of this pencil, and let $\tilde{X}_0 = X \times_Y \tilde{Y}_0$. Here $\tilde{X}_0$ is indeed the blowing up of the indeterminacies of the pullback of the above pencil. From this process we obtain a tower of fibrations
$$
\xymatrix{
	\tilde{X}_0 \ar[r]^{f_0} \ar@/^1.5pc/[rr]^{\pi_0} & \tilde{Y}_0 \ar[r] & \PP^1
	}
$$
where a general fiber $X_1$ of $\pi_0$ is isomorphic to $f^*H$. Write $Y_1 = f_0(X_1)$. 

According to Theorem \ref{filtration}, after replacing $\tilde{X}_0$ by an appropriate blowing up, there is a filtration of nef divisors
$$
L_0 > L_1 > \cdots > L_N \ge 0
$$
on $\tilde{X}_0$ satisfying the conditions therein. Here $L_0$ denotes the pullback of $L$ via the above sequence of blow-ups $\tilde{X}_0 \to X$. We also denote by $\tilde{H}_0$ the pullback of $H$ via $\tilde{Y}_0 \to Y$.

\subsection{Second step} Write $Y_1 = f_0(X_1)$. Now we do the same operation for the fibration $f_0|_{X_1}: X_1 \to Y_1$ as in \S \ref{firststep}. By abuse of the notation, $C$ still denotes a general fiber of $f_0|_{X_1}$. 

Under the above notation, we write $H_1 = \tilde{H}_0|_{Y_1}$. Choose a very general pencil in $|H_1|$. Similar to \S \ref{firststep}, let $\tilde{Y}_1$ be the blowing up of the indeterminacies of this pencil, and let $\tilde{X}_1 = X_1 \times_{Y_1} \tilde{Y}_1$. Then we obtain another tower of fibrations
$$
\xymatrix{
	\tilde{X}_1 \ar[r]^{f_1} \ar@/^1.5pc/[rr]^{\pi_1} & \tilde{Y}_1 \ar[r] & \PP^1
}
$$
where a general fiber $X_2$ of $\pi_1$ is isomorphic to $f_1^* H_1$. 

Apply Theorem \ref{filtration} again. We know that after replacing $\tilde{X}_1$ by an appropriate blowing up,  for any $0 \le i_1 \le N$, there is a filtration of nef divisors
$$
L_{i_10} > L_{i_11} > \cdots > L_{i_1N_{i_1}} \ge 0
$$
on $\tilde{X}_1$ satisfying the conditions in Theorem \ref{filtration}. Here $L_{i_10}$ denotes the pullback of the divisor $L_{i_1}|_{X_1}$ via the composition of blow-ups $\tilde{X}_1 \to X_1$. Similar to the above, let $\tilde{H}_1$ be the pullback of $H_1$ via $\tilde{Y}_1 \to Y_1$.

\subsection{Intermediate steps} 
Write $Y_2 = f_1(X_2)$, and we can keep doing the operation as in \S \ref{firststep} repeatedly. Since $\dim \tilde{X}_i = n - i < \dim \tilde{X}_{i-1}$, the whole process must terminate. Finally, we obtain $n-2$ towers of fibrations
$$
\xymatrix{
	\tilde{X}_i \ar[r]^{f_i} \ar@/^1.5pc/[rr]^{\pi_i} & \tilde{Y}_i \ar[r] & \PP^1
}
$$
for $0 \le i \le n-3$. 

\subsection{The last step} \label{laststep} 
Notice that $\dim \tilde{Y}_{n-3} = 2$, i.e., $\tilde{Y}_{n-3}$ is a fibered surfaces over $\mathbb{P}^1$ whose general fiber $Y_{n-2}$ is isomorphic to $H_{n-3}$. Let $X_{n-2} = f^*_{n-3}Y_{n-2}$. Therefore, we have the last fibration being
$$
\pi_{n-2} : (\tilde{X}_{n-2}: =) X_{n-2} \longrightarrow Y_{n-2} (=: \tilde{Y}_{n-2}),
$$
and this is a surface fibration. Recall that from the whole process, we also obtain a bunch of nef divisors $L_{i_1 i_2 \cdots i_{n-2}}$ on $\tilde{X}_{n-3}$. Since the number of these nef divisors is finite, by Theorem \ref{filtration}, replacing $\tilde{X}_{n-2} = X_{n-2}$ by an appropriate common blowing up, we may assume that for any multi-index $\alpha = i_1 i_2 \cdots i_{n-2}$ of length $n-2$, there is a filtration of nef divisors
$$
L_{\alpha 0} > L_{\alpha 1} > \cdots > L_{\alpha N_{\alpha}} \ge 0
$$
on $\tilde{X}_{n-2}$ which is similar to all the above. Here  $L_{\alpha 0}$ denotes the pullback of $L_{\alpha}$ via the (composition of) morphism $\tilde{X}_{n-2} \to \tilde{X}_{n-3}$. 

By abuse of the notation, $C$ still denotes a general fiber of $\pi_{n-2}$. Sometimes we denote $X_{n-1} = C$ in order to keep the coherence with the notation in this section.

\subsection{Summarization} \label{tree} Before we go further, let us stop here and summarize the data obtained from the whole process. As we have said before, the whole process gives rise to $n-2$ towers of fibrations $\pi_0, \cdots, \pi_{n-3}$ plus a surface fibration $\pi_{n-2}$. From those fibrations we obtain a tree-like filtration as follows:

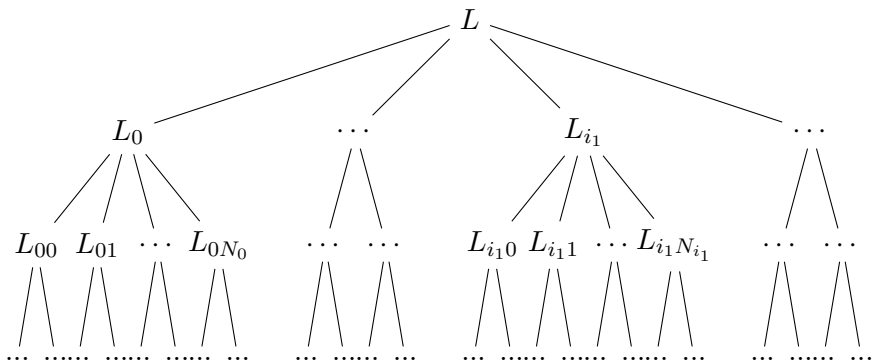
\begin{center}
	\begin{tikzpicture}
	[level 1/.style={sibling distance=30mm},
	level 2/.style={sibling distance=8mm},
	level 3/.style={sibling distance=5mm}]
	\node {$L$}
	child {node {$L_0$}
		child {node {$L_{00}$}
			child {node {...}}
			child {node {...}}
			}
		child {node {$L_{01}$}
			child {node {...}}
			child {node {...}}
		    }
		child {node {$\cdots$}
			child {node {...}}
			child {node {...}}
			}
		child {node {$L_{0N_0}$}
			child {node {...}}
			child {node {...}}
			}
		}
	child {node {$\cdots$}
		child {node {$\cdots$}
			child {node {...}}
			child {node {...}}
			}
		child {node {$\cdots$}
			child {node {...}}
			child {node {...}}
			}
		}
	child {node {$L_{i_1}$}
		child {node {$L_{i_10}$}
			child {node {...}}
			child {node {...}}
			}
		child {node {$L_{i_11}$}
			child {node {...}}
			child {node {...}}
			}
		child {node {$\cdots$}
			child {node {...}}
			child {node {...}}
			}
		child {node {$L_{i_1N_{i_1}}$}
			child {node {...}}
			child {node {...}}
			}
		}
	child {node {$\cdots$}
		child {node {$\cdots$}
			child {node {...}}
			child {node {...}}
			}
		child {node {$\cdots$}
			child {node {...}}
			child {node {...}}
			}
		};
	\end{tikzpicture}
\end{center}

More importantly, it gives rise to a set of nef thresholds
$$
a_{i_1 \cdots i_m} := \nti_{\pi_{m-1}}(L_{i_1 \cdots i_m})
$$
for any $1 \le m \le n-1$. For any multi-index $i_1 \cdots i_m$ of length $m$, we also write
$$
L'_{i_1 \cdots i_m} = L_{i_1 \cdots i_m} - (a_{i_1 \cdots i_m}-1) X_{m}.
$$
Keep in mind that $L'_{i_1 \cdots i_m}$ is a nef divisor on $\tilde{X}_{m-1}$.

All above notation in this section will be used throughout this paper.

\section{More inequalities} \label{moreinequality}
In this section, based on the tree-like filtration constructed in Section \ref{iteratedfiltration}, we prove some inequalities which will be used to prove the main results.

\subsection{Numerical inequalities from horizontal base loci} \label{horizontalbaseloci}

Let $f: X \to Y$ and $L$ be as in Section \ref{setting}. According to Theorem \ref{filtration} and Section \ref{firststep}, we obtain a tower of fibrations
$$
\xymatrix{
	\tilde{X}_0 \ar[r]^{f_0} \ar@/^1.5pc/[rr]^{\pi_0} & \tilde{Y}_0 \ar[r] & \PP^1
}
$$
with $X_1$ as a general fiber of $\pi_0$. This gives a filtration of nef divisors
$$
L_0 > L_1 > \cdots > L_N \ge 0
$$
as well as a set of base loci $\{Z_1, \cdots, Z_N\}$ on $\tilde{X}_0$. Recall that $f_0$ is a fibration of curves with a general fiber $C$. Let $1 \le \lambda_1 < \cdots < \lambda_q \le N$ be all indices such that $Z_{\lambda_j}$ has a nonzero horizontal part with respect to $f_0$, i.e.,
$$
\deg(Z_{\lambda_j}|_C) > 0.
$$
In the following, we write $Z^H_{\lambda_j}$ as the horizontal part of $Z_{\lambda_j}$ with respect to $f_0$ (see Section \ref{horizontalvertical} for the definition). We also set $\lambda_0 = 0$ and $\lambda_{q+1} = N+1$. 

\begin{prop} \label{ln-1z}
	Let $D \ge 0$ be any horizontal $\QQ$-divisor with respect to $f_0$. Then
	$$
	\left((n-1)a_{\lambda_j}L_{\lambda_j}^{n-2}X_1 + (L'_{\lambda_j})^{n-1}\right)D \ge (n-1) \sum_{i=\lambda_j}^{\lambda_{j+1}-1} a_i (L_i|_{X_1})^{n-2}(D|_{X_1})
	$$ 
	for any $0 \le j \le q$. In particular, 
	$$
	\left((n-1)a_{\lambda_j}L_{\lambda_j}^{n-2}X_1 + (L'_{\lambda_j})^{n-1}\right) Z_{\lambda_j} \ge (n-1) \sum_{i=\lambda_j}^{\lambda_{j+1}-1} a_i (L_i|_{X_1})^{n-2}(Z_{\lambda_j}^H|_{X_1}).
	$$ 
\end{prop}

\begin{proof}
	It is easy to see that the first inequality implies the second one as $Z_{\lambda_j} \ge Z_{\lambda_j}^H$. Moreover, the result is trivial if $\lambda_{j+1} - 1= \lambda_j$. Thus in the following, we may assume that $\lambda_{j+1} - 1 > \lambda_j$. Notice that under this assumption, we only need to prove that
	$$
	(L_{\lambda_j}')^{n-1} D \ge \sum_{i=\lambda_j+1}^{\lambda_{j+1}-1} a_i (L_i|_{X_1})^{n-2}(D|_{X_1}).
	$$
	
	The key point here is that for $\lambda_j < i \le \lambda_{j+1} - 1$, $Z_i$ intersects with $D$ properly as $Z_i$ is vertical with respect to $f_0$. As a result, for such $i$, we deduce from (\ref{l'i}) that 
	\begin{eqnarray*}
		(L'_{i-1})^{n-1} D & = & (L'_{i-1})^{n-2} (L'_i + a_iX_1 + Z_i)D \\
		& \ge & (L'_{i-1})^{n-2} (L'_i + a_iX_1) D \\
		& \cdots & \cdots \\
		& \ge & (L'_i + a_iX_1)^{n-1} D \\
		& = & (n-1)a_i(L_i|_{X_1})^{n-2} (D|_{X_1}) + (L'_i)^{n-1} D.
	\end{eqnarray*}
	Inductively, we obtain that
	\begin{eqnarray*}
		(L'_{\lambda_j})^{n-1} D & \ge & (n-1) \sum_{i=\lambda_j+1}^{\lambda_{j+1}-1} a_i(L_i|_{X_1})^{n-2} (D|_{X_1}) + (L'_{\lambda_{j+1}-1})^{n-1} D \\
		& \ge & (n-1) \sum_{i=\lambda_j+1}^{\lambda_{j+1}-1} a_i(L_i|_{X_1})^{n-2} (D|_{X_1}).
	\end{eqnarray*}
	Thus the whole proof is completed.
\end{proof}

\begin{prop} \label{numericalinequality'}
	Let the notation be as above. Then
	$$
	L^n \ge n \sum_{i=0}^{\lambda_1-1} a_i (L_i|_{X_1})^{n-1} + (n-1) \sum_{i = \lambda_1}^N a_i(L_{\lambda_1-1}|_{X_1})(L_i|_{X_1})^{n-2} - n (L_0|_{X_1})^{n-1}.
	$$
\end{prop}

\begin{proof}
	We have the following filtration of nef divisors
	$$
	L'_{\lambda_1 - 1} > L_{\lambda_1} > \cdots > L_N \ge 0.
	$$
	Write $a'_{\lambda_1-1}:=\nti_{f_0}(L'_{\lambda_1 - 1})$. Then it is easy to see that $a'_{\lambda_1-1} = 1$. By Proposition \ref{numericalinequality} (2) (setting $P=L_{\lambda_1-1}'$), it follows that
	\begin{eqnarray*}
		(L'_{\lambda_1 - 1})^n & \ge & (n-1) \left( a'_{\lambda_1-1} (L'_{\lambda_1-1}|_{X_1})^{n-1} +  \sum_{i = \lambda_1}^N a_i(L'_{\lambda_1-1}|_{X_1})(L_i|_{X_1})^{n-2} \right) \\
		& & - (n-1) (L'_{\lambda_1-1}|_{X_1})^{n-1} \\
		& = & (n-1) \sum_{i = \lambda_1}^N a_i(L_{\lambda_1-1}|_{X_1})(L_i|_{X_1})^{n-2}.
	\end{eqnarray*}
	Moreover, using the same proof as for Proposition \ref{numericalinequality} (1), we deduce that
	$$
	L^n - (L'_{\lambda_1 - 1})^n \ge n \sum_{i=0}^{\lambda_1-1} a_i (L_i|_{X_1})^{n-1} - n (L_0|_{X_1})^{n-1}.
	$$
	Hence the proof is completed.
\end{proof}

\subsection{Estimating nef thresholds} \label{settingp}
We need another nef divisor to estimate all nef thresholds that appear during the construction of the tree-like filtration in Section \ref{iteratedfiltration}. In this subsection, we will freely use all notation in Section \ref{iteratedfiltration}. 

Let $P \ge L$ be a nef $\QQ$-divisor on $X$ such that $\deg(P|_C) > 0$. Write $P_0 = P$ and denote by $P_m$ and $\tilde{P}_m$ the pullback of $P$ via the morphism $X_m \to X$ and $\tilde{X}_m \to X$ respectively for any $0 \le m \le n-1$. Then we have the following easy result.

\begin{prop} \label{p>l}
	For any $1 \le m \le n-1$ and any multi-index $i_1 \cdots i_m$ of length $m$, 
	$$
	\tilde{P}_{m-1}|_{X_m} = P_m \ge L_{i_1 \cdots i_m}|_{X_m}.
	$$
\end{prop}

\begin{prop} \label{pmintersection}
	For any $0 \le m \le n-1$, $P_m^{n-m} = P^{n-m}(f^*H)^m$.
\end{prop}

\begin{proof}
	This holds simply because according to the construction in Section \ref{iteratedfiltration}, $Y_m$ is isomorphic the complete intersection of $m$ general sections in $|H|$.
\end{proof}

We also have the following easy result.
\begin{prop} \label{ph>0}
	Suppose that $P-f^*H$ is pseudo-effective. Then
	$$
	P^n \ge P^{n-1}(f^*H) \ge \cdots \ge P(f^*H)^{n-1} = H^{n-1} \deg(P|_C) > 0. 
	$$
\end{prop}

In the rest of this subsection, we always assume that \emph{$P-f^*H$ is pseudo-effective}.

\begin{lemma} \label{error1}
	We have
	$$
	\sum_{i_1} a_{i_1} \le \frac{P^{n-1}L}{P_1^{n-1}} + 1 \le \frac{P^n}{P_1^{n-1}} + 1 \le \frac{2P^n}{P^{n-1}_1}.
	$$
\end{lemma}

\begin{proof}
	The filtration
	$$
	L_0 > L_1 > \cdots > L_N \ge 0
	$$
	yields
	$$
	L_0 \ge \left(\sum_{i_1} a_{i_1} - 1\right) X_1 + L'_N.
	$$ 
	Notice that $\tilde{P}_0 \ge L_0$ and $L'_N$ is nef. Thus by Proposition \ref{p>l},
	$$
	P^n \ge P^{n-1} L = \tilde{P}_0^{n-1}L_0 \ge \left(\sum_{i_1} a_{i_1} - 1\right) P_1^{n-1}.
	$$
	Hence the first and the second inequality hold. The third one holds because $P_1^{n-1} = P^{n-1}(f^*H) \le P^n$ by Proposition \ref{ph>0}.
\end{proof}

\begin{defi} \label{bmp}
	For any $1 \le m \le n-1$, we define
	$$
	B_m(X, H, P) := 2^{n-2} \frac{P^n}{\deg(P|_C)} \frac{(\tilde{P}_{m-1}|_{X_m})^{n-m}}{P^{n-m+1}_{m-1}}. 
	$$
\end{defi}

From Proposition \ref{p>l} and \ref{pmintersection}, it is easy to see that
$$
B_m(X, H, P) = 2^{n-2} \frac{P^n}{\deg(P|_C)} \frac{P_m^{n-m}}{P^{n-m+1}_{m-1}} = 2^{n-2} \frac{P^n}{\deg(P|_C)} \frac{P^{n-m}(f^*H)^m}{P^{n-m+1}(f^*H)^{m-1}}.
$$
This invariant is crucial to us in order to give the explicit error term appearing in the general relative Clifford inequality such as Theorem \ref{explicitclifford}.

Before moving to the next section, we would like to add a few remarks here. First, it is easy to check from the definition that
$$
B_1(X, H, P) = 1
$$
when $n=2$. The reason is that in this case, $X$ is a fibered surface and $X_1$ is a general fiber which is exactly $C$ (see Section \ref{laststep}). 

Second, for any $1 \le m \le n-2$, we can similarly define
$$
B_m(X_1, H_1, P_1) := 2^{n-3} \frac{P_1^{n-1}}{\deg(P_1|_C)} \frac{(\tilde{P}_m|_{X_{m+1}})^{n-m-1}}{P^{n-m}_m}
$$
for the restricted triple $(X_1, H_1, P_1)$.
This notion will be used when we prove the main result using induction. The following proposition is need in the next section.
\begin{prop} \label{bmxhp}
	For any $1 \le m \le n-2$, we have
	$$
	B_{m}(X_1, H_1, P_1) \sum_{i_1} a_{i_1} \le  \frac{2P^n}{P_1^{n-1}} B_{m}(X_1, H_1, P_1) = B_{m+1}(X, H, P).
	$$
\end{prop}

\begin{proof}
	The first inequality is due to Lemma \ref{error1}. The second identity simply follows from the definition and that $\deg(P|_C) = \deg(P_1|_C)$.
\end{proof}

\section{A Clifford type inequality with error terms} \label{Clifforderror}
In this section, we assume that $f: X \to Y$, $C$ and $H$ are the same as in Section \ref{setting}.
Let $\hat{P} \ge 0$ be a nef $\QQ$-divisor on $X$ such that
\begin{itemize}
	\item $\hat{P}-f^*H$ is pseudo-effective;
	\item $\hat{P}$ is numerically $f$-special with $d = \deg(\hat{P}|_C) > 0$.
\end{itemize}
Let $P \le \hat{P}$ be any nef $\QQ$-divisor on $X$ with $P|_C = \hat{P}|_C$. Let $L \le P$ be a nef divisor on $X$ such that $|L|$ is base point free. Therefore, $L$ is also numerically $f$-special. All these assumptions will be used throughout this section.

The following theorem is the main result in this section.
\begin{theorem} \label{mainwitherror}
	Let the notation be as above and fix a divisor $H$ as in Section \ref{setting}. Then we have
	$$
	h^0(X, L) \le \left(\frac{1}{2n!} + \frac{n-\varepsilon}{n!d}\right) PL^{n-1} + \frac{d+2}{2} \sum_{m=1}^{n-1} B_m(X, H, \hat{P}).
	$$
	Here $\varepsilon = \frac{1}{2}$ when $d = 1$ or when $C$ is hyperelliptic, and $L$ is $f$-special with $\deg(L|_C) = d-1$. Otherwise, $\varepsilon = 1$.
\end{theorem}
The number $B_m(X, H, \hat{P})$ $(m=1, \cdots, n-1)$ here is the same as before. The whole section is devoted to the proof of this theorem. From now on, $H$ is fixed. Associated to $H$ is a set of nef thresholds $a_{i_1 \cdots i_m}$ as in Section \ref{tree}, which we will keep using throughout this section.

\begin{remark} \label{beingspecial}
	Notice that if $C$ is hyperelliptic and $L$ is $f$-special, then $L|_C$ is linearly equivalent to $l g_2^1$ for some integer $0 \le l \le g-1$. Furthermore, if $d=\deg(P|_C) = \deg(L|_C) + 1$, then $P|_C$ is also special, i.e., $P$ is $f$-special.
\end{remark}

\subsection{Comparison between $h^0(X, L)$ and $L^n$}
In this subsection, we will compare $h^0(X, L)$ and $L^n$ together with other terms. 

\subsubsection{Hyperelliptic case} We first consider the case when $C$ is hyperelliptic.

\begin{prop} \label{h0llnhyper}
	Suppose that $C$ is hyperelliptic and $L$ is $f$-special. Let $D \ge 0$ be a horizontal $\QQ$-divisor on $X$ with respect to $f$ such that $\deg(D|_C) > 0$ and $D + L \le P$. Then for any integer $n_0 \ge n$, we have
	\begin{eqnarray*}
			h^0(X, L) - \frac{L^n}{2n!} - \frac{L^{n-1}D}{2(n-1)!n_0} & \le &  \frac{n_0 - \varepsilon}{n_0} \sum_{i_1, \cdots, i_{n-1}} a_{i_1} \cdots a_{i_1 \cdots i_{n-1}} \\ 
			& & + \frac{d}{2} \sum_{m=1}^{n-1} B_m(X, H, \hat{P}).
	\end{eqnarray*}
	Here $\varepsilon = \frac{1}{2}$ when $\deg(D|_C) = 1$, and $\varepsilon = 1$ otherwise.
\end{prop}

\begin{proof}
	The proof is by induction. We first consider the case when $n=2$, and thus $n_0 \ge 2$. By \cite[Theorem 3.4 (1), (3)]{Zhang_Slope}, we have
	$$
	h^0(X, L) - \frac{L^2}{4} \le a_0 + \frac{1}{2} \sum_{i_1 > 0} a_{i_1} + \frac{LC}{2}.
	$$
	Using the fact that $L = L_0 = L'_0 + (a_0-1)C$ with $L'_0$ nef, we deduce that
	\begin{eqnarray*}
		h^0(X, L) - \frac{L^2}{4} - \frac{LD}{2n_0} & \le & a_0 + \frac{1}{2} \sum_{i_1 > 0} a_{i_1} + \frac{LC}{2} - \frac{L'_0D + (a_0-1)DC}{2n_0} \\
		& \le & \left(1 - \frac{DC}{2n_0}\right) a_0 + \frac{1}{2} \sum_{i_1 > 0} a_{i_1} + \frac{LC}{2} + \frac{DC}{4} \\
		& \le & \left(\frac{n_0 - \varepsilon}{n_0} \right) a_0 + \frac{1}{2} \sum_{i_1 > 0} a_{i_1} + \frac{PC}{2} \\
		& \le & \frac{n_0 - \varepsilon}{n_0} \sum_{i_1} a_{i_1} + \frac{d}{2} B_1(X, H, \hat{P}),
	\end{eqnarray*}
	where $B_1(X, H, \hat{P}) = 1$. Thus the proof for $n=2$ is completed.
	
	In the following, we assume that Proposition \ref{h0llnhyper} holds up to dimension $n-1$. Suppose now that $\dim X = n$. Then $n_0 \ge n$. As is described in Section \ref{iteratedfiltration}, by an appropriate blowing up of $X$, we obtain a tower of fibrations
	$$
	\xymatrix{
		\tilde{X}_0 \ar[r]^{f_0} \ar@/^1.5pc/[rr]^{\pi_0} & \tilde{Y}_0 \ar[r] & \PP^1
	}
	$$
	with $X_1$ a general fiber of $\pi_0$. Let
	$$
	L_0 > L_1 > \cdots > L_N \ge 0
	$$
	be the filtration therein. Let $Z_1, \cdots, Z_N$ be the corresponding base loci. Since $C$ is hyperelliptic and the linear system $|L_{i_1}|_C|$ are special and base point free for any $1 \le i_1 \le N$, we deduce that $\deg(Z_{i_1}|_C)$ must be even for any $i_1$. Here we resume the notation in Section \ref{horizontalbaseloci}. Let $1 \le \lambda_1 < \cdots < \lambda_q \le N$ denote all indices such that $\deg(Z_{\lambda_j}|_C) > 0$ and write $Z_{\lambda_j}^H$ as the horizontal part of $Z_{\lambda_j}$ with respect to $f_0$. Then $\deg(Z_{\lambda_j}^H|_C) \ge 2$. 
	
	By Proposition \ref{h0}, Proposition \ref{numericalinequality} (1) and Proposition \ref{ln-1z}, we have
	\begin{eqnarray}
		h^0(X, L) - \frac{L^n}{2n!} & \le & \sum_{i_1} a_{i_1} \left(h^0(X_1, L_{i_1}|_{X_1}) - \frac{(L_{i_1}|_{X_1})^{n-1}}{2(n-1)!} \right) + \frac{L_0^{n-1}X_1}{2(n-1)!}  \nonumber \\
		& & - \frac{1}{2n!} \sum_{i_1 > 0} \left((n-1)a_{i_1}L^{n-2}_{i_1} X_1 +  (L'_{i_1})^{n-1} \right) Z_{i_1} \nonumber \\
		& \le & \sum_{i_1} a_{i_1} \left(h^0(X_1, L_{i_1}|_{X_1}) - \frac{(L_{i_1}|_{X_1})^{n-1}}{2(n-1)!} \right) + \frac{(L_0|_{X_1})^{n-1}}{2(n-1)!} \label{rough} \\
		& & - \sum_{j=1}^{q} \sum_{i_1 = \lambda_j}^{\lambda_{j+1} - 1} \frac{a_{i_1}(L_i|_{X_1})^{n-2}(Z^H_{\lambda_j}|_{X_1})}{2(n-2)!n}. \nonumber
	\end{eqnarray}
	Let $P_1$ and $\hat{P}_1$ denote the pullback of $P$ and $\hat{P}$ via the morphism $X_1 \to X$. Let $H_1$ denote the pull back of $H$ via the morphism $Y_1 \to Y$. Then it is easy to check that 
	\begin{itemize}
		\item $\hat{P}_1 - f_0^* H_1$ is pseudo-effective on $X_1$;
		\item $\hat{P}_1|_C = \hat{P}|_C$ so that $\hat{P}_1$ is a numerically $f$-special divisor on $X_1$ with $\deg(\hat{P}_1|_C) = d > 0$.
	\end{itemize}
	Moreover, for each $i_1$, we have $P_1 \le \hat{P}_1$, $P_1|_C =  \hat{P}_1|_C$, and $L_{i_1} \le P_1$.
	
	For any $i_1 \ge \lambda_1$, we may assume that
	$\lambda_j \le i_1 \le \lambda_{j+1} - 1$ for some $1 \le j \le q$. By the induction assumption and using the fact that $\deg(Z^H_{\lambda_j}|_C) \ge 2$, we know that
	\begin{eqnarray}
		S_{i_1} & := & h^0(X_1, L_{i_1}|_{X_1}) - \frac{(L_{i_1}|_{X_1})^{n-1}}{2(n-1)!} - \frac{(L_{i_1}|_{X_1})^{n-2}(Z^H_{\lambda_j}|_{X_1})}{2(n-2)!n} \nonumber \\
		& \le & \frac{n-1}{n}\sum_{i_2, \cdots, i_{n-1}} a_{i_1i_2} \cdots a_{i_1i_2 \cdots i_{n-1}} + \frac{d}{2} \sum_{m=1}^{n-2} B_m(X_1, H_1, \hat{P}_1). \label{i1>lambda1}
	\end{eqnarray}
	Let $\tilde{D}_0$ be the pullback of $D$ via $\tilde{X}_0 \to X$. Then $\tilde{D}_0$ is also horizontal with respect to $f_0$. By our assumption on $D$, we have $\deg(\tilde{D}_0|_C) > 0$. Similar to the above, we deduce that for any $i_1 < \lambda_1$, 
	\begin{eqnarray}
		S_{i_1} & := & h^0(X_1, L_{i_1}|_{X_1}) - \frac{(L_{i_1}|_{X_1})^{n-1}}{2(n-1)!} - \frac{(L_{i_1}|_{X_1})^{n-2}(\tilde{D}_0|_{X_1})}{2(n-2)!n_0} \nonumber \\
		& \le & \frac{n_0 - \varepsilon}{n_0}\sum_{i_2, \cdots, i_{n-1}} a_{i_1i_2} \cdots a_{i_1i_2 \cdots i_{n-1}} + \frac{d}{2} \sum_{m=1}^{n-2} B_m(X_1, H_1, \hat{P}_1). \label{i1<lambda1}
	\end{eqnarray}
	
	On the other hand, notice that $L_0 = L'_0 + (a_0-1)X_1$. By Proposition \ref{ln-1z}, we also have
	\begin{eqnarray}
		\frac{L_0^{n-1} \tilde{D}_0}{2(n-1)!n_0} 
		& \ge & \frac{\left((n-1)(a_0-1)L_0^{n-2}X_1 + (L'_0)^{n-1} \right) \tilde{D}_0}{2(n-1)!n_0}   \nonumber  \\
		& \ge & \sum_{i_1=0}^{\lambda_1 - 1} \frac{a_{i_1} (L_{i_1}|_{X_1})^{n-2}(\tilde{D}_0|_{X_1})}{2(n-2)!n_0}  - \frac{(L_0|_{X_1})^{n-2} (\tilde{D}_0|_{X_1})}{2(n-2)!n_0}. \label{ln-1d}
	\end{eqnarray}
	
	Notice that $L^{n-1}D = L_0^{n-1} \tilde{D}_0$ and in any case we have
	$$
	\frac{n-1}{n} \le \frac{n_0 - \varepsilon}{n_0}.
	$$
	Combine (\ref{rough}), (\ref{i1>lambda1}), (\ref{i1<lambda1}) and (\ref{ln-1d}) together. Thus it follows that
	\begin{eqnarray*}
		& & h^0(X, L) - \frac{L^n}{2n!} - \frac{L^{n-1}D}{2(n-1)!n_0} \\
		& \le & \sum_{i_1} a_{i_1} S_{i_1} +  \frac{(L_0|_{X_1})^{n-1}}{2(n-1)!} + \frac{(L_0|_{X_1})^{n-2}(\tilde{D}_0|_{X_1})}{2(n-2)!n_0}\\
		& \le & \frac{n_0 - \varepsilon}{n_0} \sum_{i_1, \cdots, i_{n-1}} a_{i_1} \cdots a_{i_1 \cdots i_{n-1}} + Err(X, H, L), 
	\end{eqnarray*}
	where
	\begin{eqnarray*}
		Err(X, H, L) & = &\frac{d}{2} \sum_{i_1} a_{i_1} \sum_{m=1}^{n-2} B_{m}(X_1, H_1, \hat{P}_1) +  \frac{(L_0|_{X_1})^{n-1}}{2(n-1)!} + \frac{(L_0|_{X_1})^{n-2}(\tilde{D}_0|_{X_1})}{2(n-2)!n_0} \\
		& \le & \frac{d}{2} \sum_{m=2}^{n-1} B_m(X, H, \hat{P}) + \frac{P_1^{n-1}}{2(n-1)!} \\
		& \le & \frac{d}{2} \sum_{m=1}^{n-1} B_m(X, H, \hat{P}).
	\end{eqnarray*} 
	The above inequality about $Err(X, H, L)$ is from Proposition \ref{bmxhp}. Hence the whole proof is completed.
\end{proof}

A slight modification of the above proof will give the following result.

\begin{prop} \label{h0llnhyper1}
	Suppose that $C$ is hyperelliptic and $L$ is not $f$-special. Then for any integer $n_0 \ge n$, we have
	$$
	h^0(X, L) - \frac{L^n}{2n!} \le  \frac{n_0 - 1}{n_0} \sum_{i_1, \cdots, i_{n-1}} a_{i_1} \cdots a_{i_1 \cdots i_{n-1}} + \frac{d}{2} \sum_{m=1}^{n-1} B_m(X, H, \hat{P}).
	$$
\end{prop}

\begin{proof}
	The proof is also by induction. When $n=2$, this result has been proved in \cite[Page 101--102, Section 2.3]{YuanZhang_RelNoether} where the authors proved a slightly stronger result that
	$$
	h^0(X, L) - \frac{L^2}{4} \le \frac{1}{2} \sum_{i_1 > 0} a_{i_1} + \frac{d}{2}.
	$$
	
	For general $n$, we adopt the same notation as in the proof of Proposition \ref{h0llnhyper}. Moreover, since $L$ is not $f$-special, there is an integer $0 < \mu \le N+1$ such that for any $0 \le i_1 < \mu$, the divisor $L_{i_1}$ is not $f_0$-special. This implies that $\mu = \lambda_{j_0}$ for some $1 \le j_0 \le q+1$.
	
	Notice that \eqref{rough} also holds here. If $i_1 < \mu$, then $L_{i_1}|_{X_1}$ is not $f_0|_{X_1}$-special. Thus by induction, we have
	\begin{eqnarray}
		h^0(X_1, L_{i_1}|_{X_1}) - \frac{(L_{i_1}|_{X_1})^{n-1}}{2(n-1)!} & \le &  \frac{n-1}{n}\sum_{i_2, \cdots, i_{n-1}} a_{i_1i_2} \cdots a_{i_1i_2 \cdots i_{n-1}} + \nonumber \\
		& & \frac{d}{2} \sum_{m=1}^{n-2} B_m(X_1, H_1, \hat{P}_1). \label{hypernonspecial}
	\end{eqnarray}

	If $i_1 \ge \mu$, then $i_1 \ge \lambda_1$. Similar as before, we assume that $\lambda_j \le i_1 \le \lambda_{j+1}-1$ for some $1\le j \le q+1$. Then we claim that \eqref{i1>lambda1} holds here. Actually, we only need to show that $\deg(Z^H_{\lambda_j}|_C) \ge 2$. This is true when $L_{\lambda_j - 1}$ is $f_0$-special. If $L_{\lambda_j - 1}$ is not $f_0$-special, then $j=j_0$. In particular, $L_{\lambda_j} = L_\mu$ is $f_0$-special. Thus by Remark \ref{beingspecial}, $\deg(Z^H_{\lambda_j}|_C) \ge 2$.
	
	With \eqref{rough}, \eqref{i1>lambda1} and \eqref{hypernonspecial}, the rest argument can be proceeded identically as in the proof of Proposition \ref{h0llnhyper}. Hence the proof is completed. 
\end{proof}

\subsubsection{Non-hyperelliptic case} Here we assume that $C$ is non-hyperelliptic.

\begin{prop} \label{h0llnnonhyper1}
	Suppose that $C$ is non-hyperelliptic and $\deg(L|_C) = 0$. Let $D \ge 0$ be a horizontal $\QQ$-divisor on $X$ with respect to $f$ such that $\deg(D|_C) > 0$ and $D + L \le P$. Then for any integer $n_0 \ge n$, the inequality in Proposition \ref{h0llnhyper} also holds.
\end{prop}

\begin{proof}
	The proof here is similar as before. The only difference is that here $L^n=0$. However, it does not affect the whole proof. For example, when $n=2$, we have $n_0 \ge 2$. By \cite[Theorem 3.4 (1)]{Zhang_Slope} and a similar argument to the proof of Proposition \ref{h0llnhyper}, we obtain
	$$
	h^0(X, L) - \frac{LD}{2n_0} \le a_0 - \frac{LD}{2n_0} \le \left(\frac{n_0 - \varepsilon}{n_0}\right) a_0 + \frac{d}{2} B_1(X, H, \hat{P}).
	$$
	For general $n$, the proof is based on exactly the same argument involving (\ref{rough}), (\ref{i1<lambda1}) and (\ref{ln-1d}). We omit the proof here and leave it to the interested reader.
\end{proof}

\begin{prop} \label{h0llnnonhyper2}
	Suppose that $C$ is non-hyperelliptic and $0 < \deg(L|_C) < d$. Then for any integer $n_0 \ge n$, we have
	$$
	h^0(X, L) - \frac{L^n}{2n!} \le \frac{n_0-1}{n_0}\sum_{i_1, \cdots, i_{n-1}} a_{i_1} \cdots a_{i_1 \cdots i_{n-1}} \\
	+ \frac{d}{2} \sum_{m=1}^{n-1} B_m(X, H, \hat{P}).
	$$
\end{prop}

As we assume in this section that $|L|$ is base point free and $g \ge 2$, the inequality $0 < \deg(L|_C) < d$ actually implies that $2 \le \deg(L|_C) < d$.

\begin{proof}
	Since the proof here is very similar to that of Proposition \ref{h0llnhyper}, we sketch it here and only emphasize the differences. We also use the notation in Section \ref{horizontalbaseloci}.
	
	When $n=2$, $n_0 \ge 2$. By \cite[Theorem 3.4 (4)]{Zhang_Slope}, when $L$ is $f$-special, we have
	$$
	h^0(X, L) - \frac{L^2}{4} \le \frac{1}{2} a_0 + \frac{1}{4} \sum_{i_1 > 0} a_{i_1} + \frac{LC}{2} \le \frac{n_0-1}{n_0} \sum_{i_1} a_{i_1} + \frac{d}{2} B_1(X, H, \hat{P}).
	$$
	Notice that the proof of \cite[Theorem 3.4 (4)]{Zhang_Slope} actually applies to the case when $L$ is numerically $f$-special verbatim, and the same result still holds.
	
	For general $n$, we follow the same notation as in the proof of Proposition \ref{h0llnhyper}. The inequality (\ref{rough}) still holds here. If $\deg(L_{i_1}|_C) > 0$, by induction, we obtain that
	\begin{eqnarray*}
		h^0(X_1, L_{i_1}|_{X_1}) - \frac{(L_{i_1}|_{X_1})^{n-1}}{2(n-1)!} & \le & \frac{n_0-1}{n_0}\sum_{i_2, \cdots, i_{n-1}} a_{i_1i_2} \cdots a_{i_1i_2 \cdots i_{n-1}} + \\ & & 
		\frac{d}{2} \sum_{m=1}^{n-2} B_m(X_1, H_1, \hat{P}_1).
	\end{eqnarray*}
	If $\deg(L_{i_1}|_C) = 0$, then it implies that $i_1 \ge \lambda_q$. Notice that $g \ge 2$ implies $\deg(Z_{\lambda_q}^H|_C) \ge 2$ in this case. By Proposition \ref{h0llnnonhyper1}, we deduce that 
	\begin{eqnarray*}
	& & h^0(X_1, L_{i_1}|_{X_1}) - \frac{(L_{i_1}|_{X_1})^{n-1}}{2(n-1)!} - \frac{(L_i|_{X_1})^{n-2}(Z_{\lambda_q}^H|_{X_1})}{2(n-2)!n} \\
	& \le & \frac{n-1}{n}\sum_{i_2, \cdots, i_{n-1}} a_{i_1i_2} \cdots a_{i_1i_2 \cdots i_{n-1}} + \frac{d}{2} \sum_{m=1}^{n-2} B_m(X_1, H_1, \hat{P}_1) 
	\end{eqnarray*}
	in exactly the same way as for (\ref{i1>lambda1}). Then the proof can be completed via a similar way to that of Proposition \ref{h0llnhyper}.
\end{proof}

\subsection{Estimate of $PL^{n-1}$} 
In this subsection, we mainly consider the estimate of $PL^{n-1}$.
\begin{prop} \label{pln-11}
	Let the notation be as before. Then we have
	$$
	\frac{PL^{n-1}}{d (n-1)!} \ge \sum_{i_1, \cdots, i_{n-1}} a_{i_1} \cdots a_{i_1 \cdots i_{n-1}} - \sum_{m=1}^{n-1} B_m(X, H, \hat{P}).
	$$
\end{prop}

\begin{proof}
	The proof is by induction. When $n=2$, the corresponding result in this case is just Lemma \ref{error1}.
	
	In the following, we assume that Proposition \ref{pln-11} holds up to dimension $n-1$. Suppose that $\dim X = n$. Resume the notation in Section \ref{firststep} once again. Let $H_1$, $P_1$ and $\hat{P}_1$ be the same as in the proof of Proposition \ref{h0llnhyper}. Then by Proposition \ref{numericalinequality} (2),
	$$
	\frac{PL^{n-1}}{d(n-1)!} \ge \sum_{i_1} a_{i_1} \frac{P_1 (L_{i_1}|_{X_1})^{n-2}}{d (n-2)!} - \frac{P_1(L_0|_{X_1})^{n-2}}{d (n-2)!}.
	$$
	By the induction assumption,
	$$
	\frac{P_1 (L_{i_1}|_{X_1})^{n-2}}{d (n-2)!} \ge \sum_{i_2, \cdots, i_{n-1}}  a_{i_1i_2} \cdots a_{i_1i_2 \cdots i_{n-1}} - \sum_{m=1}^{n-2} B_m(X_1, H_1, \hat{P}_1).
	$$
	Combine the above two inequalities, and it follows that
	\begin{eqnarray*}
		\frac{PL^{n-1}}{d(n-1)!} & \ge &  \sum_{i_1, \cdots, i_{n-1}} a_{i_1} \cdots a_{i_1 \cdots i_{n-1}} - \sum_{i_1} a_{i_1} \sum_{m=1}^{n-2} B_m(X_1, H_1, \hat{P}_1) \\
		& &  - \frac{P_1(L_0|_{X_1})^{n-2}}{d (n-2)!}.
	\end{eqnarray*}
	Using the estimate at the end of the proof of Proposition \ref{h0llnhyper} once again, we deduce that
	$$
    \sum_{i_1} a_{i_1} \sum_{m=1}^{n-2} B_m(X_1, H_1, \hat{P}_1) + \frac{P_1(L_0|_{X_1})^{n-2}}{d (n-2)!} \le \sum_{m=1}^{n-1} B_m(X, H, \hat{P}).
	$$
	Thus the proof is completed.
\end{proof}

\subsection{Proof of Theorem \ref{mainwitherror}}
We first remind of Theorem \ref{mainwitherror} here for the convenience of the reader.
\begin{theorem} [Theorem \ref{mainwitherror}]
	Let the notation be as above. We have
	$$
	h^0(X, L) \le \left(\frac{1}{2n!} + \frac{n-\varepsilon}{n!d}\right) PL^{n-1} + \frac{d+2}{2} \sum_{m=1}^{n-1} B_m(X, H, \hat{P}).
	$$
	Here $\varepsilon = \frac{1}{2}$ when $d = 1$ or when $C$ is hyperelliptic, $L$ is $f$-special with $\deg(L|_C) = d-1$. Otherwise, $\varepsilon = 1$.
\end{theorem}

\begin{proof}
	We divide the proof into two cases with respect to $\deg(L|_C)$.
	
	\textbf{Case 1: $\deg(L|_C) < d$}. The result in this case is a combination of the previous results.
	
	First, let us assume that $C$ is hyperelliptic. As $\deg(L|_C) < d$, we deduce that $\deg\left((P-L)|_C\right) > 0$. Let $D$ be the horizontal part of $P-L$ with respect to $f$. By Proposition \ref{h0llnhyper}, \ref{h0llnhyper1} (taking $n_0 = n$) and \ref{pln-11}, it follows that
	\begin{eqnarray*}
		h^0(X, L) - \frac{PL^{n-1}}{2n!} & \le & h^0(X, L) - \frac{L^n}{2n!} - \frac{L^{n-1}D}{2n!} \\
		& \le & \frac{n-\varepsilon}{n} \sum_{i_1, \cdots, i_{n-1}} a_{i_1} \cdots a_{i_1 \cdots i_{n-1}} + \frac{d}{2} \sum_{m=1}^{n-1} B_m(X, H, \hat{P}) \\
		& \le & \left( \frac{n-\varepsilon}{n} \right) \frac{PL^{n-1}}{d (n-1)!} + \left(\frac{n-\varepsilon}{n} + \frac{d}{2}\right) \sum_{m=1}^{n-1} B_m(X, H, \hat{P})\\
		& \le & \left( \frac{n-\varepsilon}{n! d} \right) PL^{n-1} + \frac{d+2}{2} \sum_{m=1}^{n-1} B_m(X, H, \hat{P}).
	\end{eqnarray*}
	
	Second, assume that $C$ is non-hyperelliptic. If $\deg(L|_C) = 0$, then by Proposition \ref{h0llnnonhyper1} (taking $n_0 = n$) and \ref{pln-11}, the above proof applies here verbatim. If $\deg(L|_C) > 0$, then by Proposition \ref{h0llnnonhyper2} and \ref{pln-11}, similar to the above, we obtain that
	\begin{eqnarray*}
		h^0(X, L) - \frac{PL^{n-1}}{2n!} & \le & h^0(X, L) - \frac{L^n}{2n!} \\
		& \le & \frac{n-1}{n} \sum_{i_1, \cdots, i_{n-1}} a_{i_1} \cdots a_{i_1 \cdots i_{n-1}} + \frac{d}{2} \sum_{m=1}^{n-1} B_m(X, H, \hat{P}) \\
		& \le & \left(\frac{n-1}{n! d}\right) PL^{n-1} + \frac{d+2}{2} \sum_{m=1}^{n-1} B_m(X, H, \hat{P}).
	\end{eqnarray*}
	
	\textbf{Case 2: $\deg(L|_C) = d$}. Notice that in this case $|\hat{P}|_C|$ is base point free. In particular, $d > 1$. The proof in this case is by induction. When $n=2$, the result is implied by \cite[Theorem 1.2]{YuanZhang_RelNoether} which states that
	$$
	h^0(X, L) \le \left(\frac{1}{4} + \frac{1}{2d}\right)L^2 + \frac{d+2}{2}.
	$$
	Since $L^2 \le PL$, the result for $n=2$ is verified. In the following, we assume that Theorem \ref{mainwitherror} for $\deg(L|_C) = d$ is true up to dimension $n-1$.
	
	Suppose that $\dim X = n$. Resume the notation in Section \ref{iteratedfiltration} and Section \ref{horizontalbaseloci}. Then (\ref{rough}) still holds here. That is,
	\begin{eqnarray}
	h^0(X, L) - \frac{L^n}{2n!} & \le & \sum_{i_1} a_{i_1} \left(h^0(X_1, L_{i_1}|_{X_1}) - \frac{(L_{i_1}|_{X_1})^{n-1}}{2(n-1)!} \right) - \frac{(L_0|_{X_1})^{n-1}}{2(n-1)!} \nonumber \\
	& & - \sum_{j=1}^{q} \sum_{i_1 = \lambda_j}^{\lambda_{j+1} - 1} \frac{a_{i_1}(L_i|_{X_1})^{n-2}(Z_{\lambda_j}^H|_{X_1})}{2(n-2)!n}. \label{roughd}
	\end{eqnarray}
	
	Let $\lambda_1$ be the smallest index such that $\deg(Z_{\lambda_1}|_C) > 0$. When $i_1 < \lambda_1$, $\deg(L_{i_1}|_C) = \deg(L|_C) = d$. Thus by the induction assumption, we obtain
	\begin{eqnarray}
	S'_{i_1} & := & h^0(X_1, L_{i_1}|_{X_1}) - \frac{(L_{i_1}|_{X_1})^{n-1}}{2(n-1)!}  \nonumber \\
	& \le &  \frac{(n-2)(L_{i_1}|_{X_1})^{n-1}}{(n-1)!d} + \frac{d+2}{2} \sum_{m=1}^{n-2} B_m(X_1, H_1, \hat{P}_1) \label{i1<lambda1d} \\
	& \le &
	\frac{(L_{i_1}|_{X_1})^{n-1}}{(n-2)!d} + \frac{d+2}{2} \sum_{m=1}^{n-2} B_m(X_1, H_1, \hat{P}_1). \nonumber
	\end{eqnarray}
	Here $H_1$ and $\hat{P}_1$ are the same as those in the proof of Proposition \ref{h0llnhyper}.
	
	When $i_1 \ge \lambda_1$, $\deg(L_{i_1}|_C) < d$. Similar to previous proofs, we know that $\lambda_j \le i_1 \le \lambda_{j+1}-1$ for some $1 \le j \le q$. By Proposition \ref{h0llnhyper},  \ref{h0llnhyper1}, \ref{h0llnnonhyper1} and \ref{h0llnnonhyper2}, we obtain
	\begin{eqnarray}
		S'_{i_1} & := & h^0(X_1, L_{i_1}|_{X_1}) - \frac{(L_{i_1}|_{X_1})^{n-1}}{2(n-1)!} - \frac{(L_{i_1}|_{X_1})^{n-2}(Z^H_{\lambda_j}|_{X_1})}{2(n-2)!n} \nonumber \\
		& \le & \frac{n-1}{n} \sum_{i_2, \cdots, i_{n-1}} a_{i_1i_2} \cdots a_{i_1i_2 \cdots i_{n-1}} + \frac{d}{2} \sum_{m=1}^{n-2} B_m(X_1, H_1, \hat{P}_1). \label{i1>lambda1d1}
	\end{eqnarray}
	The reason why we use $\frac{n-1}{n}$ in (\ref{i1>lambda1d1}) instead of $\frac{n-\varepsilon}{n}$ is the following: if $C$ is hyperelliptic, then as we have seen before, either $L_{i_1}|_{X_1}$ is not $f_0|_{X_1}$-special, or $\deg(Z^H_{\lambda_j}|_C) \ge 2$ for any $1 \le j \le q$. Therefore, we can always take $\varepsilon = 1$ when applying Proposition \ref{h0llnhyper} or \ref{h0llnhyper1}. If $C$ is non-hyperelliptic, then $\deg(Z^H_{\lambda_q}|_C) \ge 2$ so that we can take $\varepsilon = 1$ when applying Proposition \ref{h0llnnonhyper1}. Notice that $\deg(L'_{\lambda_1-1}|_C) = d$ and $\hat{P}_1 \ge L'_{\lambda_1-1}|_{X_1} \ge L_{i_1}|_{X_1}$. By Proposition \ref{pln-11}, we deduce that
	$$
	\frac{(L'_{\lambda_1-1}|_{X_1})(L_{i_1}|_{X_1})^{n-2}}{d(n-2)!} \ge \sum_{i_2, \cdots, i_{n-1}} a_{i_1i_2} \cdots a_{i_1i_2 \cdots i_{n-1}} - \sum_{m=1}^{n-2} B_m(X_1, H_1, \hat{P}_1),
	$$
	which, together with (\ref{i1>lambda1d1}), implies that 
	\begin{equation}
		S'_{i_1} \le \left(\frac{n-1}{n}\right) \frac{(L'_{\lambda_1-1}|_{X_1})(L_{i_1}|_{X_1})^{n-2}}{d(n-2)!} + \frac{d+2}{2} \sum_{m=1}^{n-2} B_m(X_1, H_1, \hat{P}_1). \label{i1>lambda1d}
	\end{equation}
	
	Combine (\ref{roughd}), (\ref{i1<lambda1d}) and (\ref{i1>lambda1d}) together. Apply Proposition \ref{numericalinequality'}. It follows that
	\begin{eqnarray*}
		h^0(X, L) - \frac{L^n}{2n!} & \le & \sum_{i_1} a_{i_1} S'_{i_1} + \frac{(L_0|_{X_1})^{n-1}}{2(n-1)!} \\
		& \le & \frac{n-1}{n} \left(\sum_{i_1 < \lambda_1} a_{i_1} \frac{n(L_{i_1}|_{X_1})^{n-1}}{(n-1)!d} + \sum_{i_1 \ge \lambda_1} a_{i_1} \frac{(L'_{\lambda_1-1}|_{X_1})(L_{i_1}|_{X_1})^{n-2}}{d(n-2)!} \right) \\
		& & + \frac{d+2}{2} \sum_{i_1} a_{i_1} \sum_{m=1}^{n-2} B_m(X_1, H_1, \hat{P}_1) + \frac{(L_0|_{X_1})^{n-1}}{2(n-1)!}  \\
		& \le & \frac{(n-1)L^n}{n!d} + Err'(X, H, L).
	\end{eqnarray*}
	Here
	$$
	Err'(X, H, L) :=\frac{d+2}{2} \sum_{i_1} a_{i_1} \sum_{m=1}^{n-2} B_m(X_1, H_1, \hat{P}_1) + \frac{(L_0|_{X_1})^{n-1}}{2(n-1)!} + \frac{(L_0|_{X_1})^{n-1}}{d(n-2)!}.
	$$
	Using the same technique as in the proof of Proposition \ref{h0llnhyper}, it is easy to check that
	$$
	Err'(X, H, L) \le \frac{d+2}{2} \sum_{m=1}^{n-1} B_m(X, H, \hat{P}).
	$$
	Thus the proof is completed.
\end{proof}

\section{Proof of Theorem \ref{explicitclifford}} \label{proofClifford}

With Theorem \ref{mainwitherror}, we can easily prove Theorem \ref{explicitclifford}.

\begin{theorem} [Theorem \ref{explicitclifford}] \label{explicitestimate}
	Let $f: X \to Y$ be a fibration between two normal varieties $X$ and $Y$ with $\dim X = n \ge 2$ and with general fiber $C$ a smooth curve of genus $g \ge 2$. Fix a base-point-free linear system $|H|$ on $Y$ with $H^{n-1} > 0$. Let $L$ be a nef and numerically $f$-special divisor on $X$ with $d = \deg(L|_C) > 0$.
	\begin{itemize}
		\item [(1)] If $L-f^*H$ is pseudo-effective, then
		$$
		\displaystyle h^0(X, \CO_X(L)) \le \left(\frac{1}{2n!} + \frac{n-\varepsilon}{n! d}\right) L^n + \frac{d+2}{2} \sum_{m=1}^{n-1} B_m(X, H, L).
		$$
		Here $\varepsilon = \frac{1}{2}$ when $d = 1$ or when $C$ is hyperelliptic, $L$ is $f$-special, and $d$ is odd. Otherwise, $\varepsilon = 1$.
		\item [(2)] If $L-f^*H$ is not pseudo-effective, then 
		$$
		h^0(X, \CO_X(L)) \le h^0(X_H, \CO_{X_H}(L)). 
		$$
		Here $X_H = f^*H$ is viewed as a subvariety of $X$.
	\end{itemize}
\end{theorem}

\begin{proof}
	If $L-f^*H$ is not pseudo-effective, then $h^0(X, \CO_X(L-f^*H)) = 0$ and we will obtain the inequality (2).
	
	If $L-f^*H$ is pseudo-effective, in order to apply Theorem \ref{mainwitherror}, we consider the following diagram:
	$$
	\xymatrix{
		X' \ar[r]^{\pi_X}  \ar[d]_{f'} & X \ar[d]^f \\
		Y' \ar[r]^{\pi_Y} & Y 
	}
	$$
	where $\pi_X: X' \to X$ and $\pi_Y: Y' \to Y$ are respectively resolution of singularities of $X$ and $Y$, and $f'$ is the induced fibration by $f$. Let $L' = \pi^*_X L$ and $H'= \pi^*_Y H$. For simplicity, we still use $C$ to denote a general fiber of $C$. Notice that here $L'$ is a $\QQ$-divisor in general. Replacing $X'$ by an appropriate blowing up, we may assume that 
	$$
	L' = M + Z,
	$$
	where $Z \ge 0$, $|M|$ is base point free, and $h^0(X', \rounddown{L'}) = h^0(X', M)$.
	
	It is easy to see that 
	\begin{itemize}
		\item $L'-f'^*H'$ is pseudo-effective;
		\item $L'$ is numerically $f'$-special with $d = \deg(L'|_C) > 0$.
	\end{itemize}   
	Thus by Theorem \ref{mainwitherror}, we obtain
	$$
	h^0(X', M) \le \left(\frac{1}{2n!} + \frac{n-\varepsilon}{n! d}\right) L'M^{n-1} + \frac{d+2}{2} \sum_{m=1}^{n-1} B_m(X', H', L').
	$$
	Here $\varepsilon = \frac{1}{2}$ when (i): $d = 1$ or (ii): $C$ is hyperelliptic and $\deg (M|_C) = d - 1$. Otherwise $\varepsilon = 1$. In fact, (ii) implies that $d$ is odd.

	To finish the whole proof, notice that we have $h^0(X', M) = h^0(X, \CO_X(L))$, $L'M^{n-1} \le (L')^n = L^n$, and 
	\begin{eqnarray*}
		B_m(X', H', L') & = & 2^{n-2} \frac{(L')^n}{\deg(L'|_C)} \frac{(L')^{n-m} (f'^*H')^m}{(L')^{n-m+1}(f'^*H')^{m-1}} \\
		& = & 2^{n-2} \frac{L^n}{d} \frac{L^{n-m} (f^*H)^m}{L^{n-m+1}(f^*H)^{m-1}} \\
		& = & B_m(X, H, L) 
	\end{eqnarray*}
    by Definition \ref{bmp}. Hence the proof is completed.
\end{proof}

\section{Continuous rank over abelian varieties}

In this section, we study some basic properties of the continuous rank of divisors and coherent sheaves.

Suppose that $a: X \to A$ is a nontrivial morphism from a normal variety $X$ to an abelian variety $A$. For any coherent sheaf $\CE$ on $X$, similar to Definition \ref{continuousrank}, we define
$$
h^0_a(X, \CE)  := \min \{h^0(X, \CE \otimes a^* \alpha) |  \alpha \in \Pic^0(A)\}.
$$
For any divisor $L$ on $X$, we have
$$
h^0_a(X, \CO_X(L)) := \min \{h^0 (X, \CO_X(L) \otimes a^* \alpha) |  \alpha \in \Pic^0(A)\}
$$
by Definition \ref{continuousrank}. It is very easy to deduce from the above definition that for any $\alpha \in \Pic^0(A)$, we always have
$$
h^0_a(X, \CE) = h^0_a(X, \CE \otimes a^* \alpha).
$$
Moreover, if $L \ge L'$, then
$$
h^0_a(X, \CO_X(L)) \ge h^0_a(X, \CO_X(L')).
$$

For any integer $k > 0$, denote by $[k]: A \to A$ the multiplication map of $A$ by $k$. Write $X_{[k]} : = X \times_{[k]} A$. We have the following commutative diagram:
$$
\xymatrix{
	X_{[k]} \ar[r]^{\phi_{[k]}}  \ar[d]_{a_{[k]}} & X \ar[d]^a \\
	A \ar[r]^{[k]} & A 
	}
$$
Notice that if $a(X)$ generates $A$, then $X_{[k]}$ is irreducible.

We have the following proposition.

\begin{prop} \label{limith0}
	Suppose that $a: X \to A$ is a nontrivial morphism from a normal variety $X$ to an abelian variety $A$ of dimension $m$ such that $a(X)$ generates $A$. Then for any coherent sheaf $\CE$ on $X$,
	$$
	h^0_a(X, \CE) = \lim\limits_{k \to \infty} \frac{h^0(X_{[k]}, \phi_{[k]}^* \CE)}{k^{2m}}.
	$$
\end{prop}

\begin{proof}
	By the projection formula, we know that for any $k > 0$,
	$$
	h^0(X_{[k]}, \phi_{[k]}^* \CE) = h^0(X, (\phi_{[k]})_*\phi_{[k]}^* \CE) = \sum_{\alpha \in T_k} h^0(X, \CE \otimes a^* \alpha),
	$$
	where $T_k$ is the set of all $k$-torsion elements in $\Pic^0(A)$. Consider the following subset of $T_k$:
	$$
	S_k:= \{\alpha \in T_k | h^0(X, \CE \otimes a^*\alpha) > h^0_a(X, \CE) \}.
	$$
	By the semi-continuity theorem, all these $S_k$ lie in a proper subvariety of $\Pic^0(A)$, which implies that
	$$
	\lim\limits_{k \to \infty} \frac{\# S_k}{\# T_k} = 0.
	$$
	As a result, we deduce that
	$$
	\lim\limits_{k \to \infty} \frac{h^0(X_{[k]}, \phi_{[k]}^* \CE)}{k^{2m}} = \lim\limits_{k \to \infty} \frac{\sum_{\alpha \in T_k} h^0(X, \CE \otimes a^* \alpha)}{\# T_k} = h^0_a (X, \CE).
	$$
	Thus the proof is completed.
\end{proof}

From the above proposition, we obtain the following result which says that the continuous rank behaves well under \'etale covers.
\begin{coro} \label{h0acomparison}
	Under the same assumption as Proposition \ref{limith0}, for any $k > 0$, we have
	$$
	h^0_{a_{[k]}} (X_{[k]}, \phi_{[k]}^* \CE) = k^{2m} h^0_a(X, \CE).
	$$
\end{coro}

\begin{proof}
	By Proposition \ref{limith0}, for any fixed $k>0$, we have
	\begin{eqnarray*}
		h^0_{a_{[k]}} (X_{[k]}, \phi_{[k]}^* \CE) & = & \lim\limits_{l \to \infty} \frac{h^0(X_{[lk]}, \phi_{[lk]}^* \CE)}{l^{2m}} \\
		& = & k^{2m} \lim\limits_{l \to \infty} \frac{h^0(X_{[lk]}, \phi_{[lk]}^* \CE)}{(lk)^{2m}} = k^{2m} h^0_a(X, \CE).
	\end{eqnarray*}
	Thus the result holds.
\end{proof}

\section{Proof of main results} \label{proofmain}
In this section, we prove the main results of this paper.

\subsection{Relative Clifford inequality}
Let $f: X \to Y$ be a fibration between normal varieties $X$ and $Y$ with general fiber $C$ a smooth curve of genus $g \ge 2$. Assume that $\dim X = n \ge 2$. Let $b: Y \to A$ be a nontrivial map from $Y$ to an abelian variety. Then we have the following commutative diagram:
$$
\xymatrix{
	X \ar[rd]^a \ar[d]_f & \\
	Y \ar[r]^b & A}
$$
We have the following easy lemma.
\begin{lemma} \label{h0fh0a}
	For any coherent sheaf $\CE$ on $X$,
	$$
	h^0_f(X, \CE) \le h^0_a(X, \CE).
	$$
	The equality holds when $a$ is the Albanese map of $Y$.
\end{lemma}

\begin{proof}
	Pick any $\alpha \in \Pic^0(A)$ such that $h^0_a(X, \CE) = h^0(X, \CE \otimes a^* \alpha)$. Then $\beta :=b^* \alpha \in \Pic^0(Y)$ and
	$$
	h^0_f(X, \CE) \le h^0(X, \CE \otimes f^*\beta) = h^0(X, \CE \otimes a^* \alpha).
	$$
	Thus the proof is completed. When $a$ is the Albanese map of $Y$, $\Pic^0(A) = \Pic^0(Y)$. Thus the equality follows simply from the definition.
\end{proof}

\begin{theorem} [Theorem \ref{main}] \label{clifford} 
	Suppose that $Y$ is of maximal Albanese dimension. Let $L$ be a nef and numerically $f$-special divisor on $X$ with $d = \deg(L|_C) > 0$. Then
	$$
	h^0_f(X, \CO_X(L)) \le \left(\frac{1}{2n!} + \frac{n-\varepsilon}{n!d}\right) L^n.
	$$
	Here $\varepsilon = \frac{1}{2}$ when $d = 1$ or when $C$ is hyperelliptic, $L$ is $f$-special and $d$ is odd. Otherwise, $\varepsilon = 1$.
\end{theorem}

\begin{proof}
	Let $b:=\alb_Y$ be the Albanese map of $Y$. Then we have the following diagram
	$$
	\xymatrix{
		X \ar[rd]^a \ar[d]_f & \\
		Y \ar[r]^b & A}
	$$
	with $A$ the Albanese variety of $Y$ and $\dim A \ge n-1$. Since $f$ is surjective, $a(X)$ generates $A$. Hence by Lemma \ref{h0fh0a}, it suffices to prove that 
	$$
	h^0_a(X, \CO_X(L)) \le \left(\frac{1}{2n!} + \frac{n-\varepsilon}{n!d}\right) L^n.
	$$
	
	Without loss of the generality, we assume that $h^0_a(X, \CO_X(L)) > 0$. It means that $h^0(X, \CO_X(L) \otimes a^* \alpha) > 0$ for some $\alpha \in \Pic^0(A)$. Replacing $\CO_X(L)$ by $\CO(L) \otimes a^*\alpha$, we may assume that $L \ge 0$. 
	
	For any $k > 0$, consider the following commutative diagram
	$$
	\xymatrix{
		X_{[k]} \ar[r]^{\phi_{[k]}} \ar[d]_{f_{[k]}} 
		\ar@/_2.5pc/[dd]_{a_{[k]}} & X \ar[d]^f  \ar@/^2.5pc/[dd]^a \\
		Y_{[k]} \ar[r]^{\psi_{[k]}} \ar[d]_{b_{[k]}} & Y \ar[d]^b \\
		A \ar[r]^{[k]} & A}
	$$
	Here $Y_{[k]} = Y \times_{[k]} A$ and $X_{[k]} = X \times_{[k]} A$, and they are both irreducible. Notice that here $\deg \phi_{[k]} = \deg \psi_{[k]} = \deg [k] = k^{2 \dim A}$. Write
	$$
	L_{[k]} = \phi_{[k]}^* L.
	$$
	
	Fix a sufficiently ample divisor $G$ on $A$. Write
	$$
	H = b^*G \quad \mbox{and} \quad H_{[k]} = b^*_{[k]} G.
	$$
	We may assume that each $|H_{[k]}|$ is base point free. The ampleness of $G$ guarantees that $H_{[k]}^{n-1} > 0$.
	Notice that by \cite[Chapter 2, Proposition 3.5]{Birkenhake_Lange_AbelianVar}, we have
	\begin{equation}
		k^2 H_{[k]} \sim_{\mathrm{num}} \psi_{[k]}^* H. \label{numericalequivalence}
	\end{equation}
	
	If $L - f^*H$ is pseudo-effective, then $L_{[k]} - f_{[k]}^* H_{[k]}$ is also pseudo-effective. Moreover, for any $1 \le m \le n-1$, we deduce from (\ref{numericalequivalence}) that
	$$
	L_{[k]}^{n-m} (f_{[k]}^*H_{[k]})^{m} = k^{-2m} L_{[k]}^{n-m} ((\psi_{[k]} \circ f_{[k]})^*H)^{m} = k^{2 \dim A -2m}L^{n-m}(f^*H)^m.
	$$
	This implies that
	$$ 
	\frac{L_{[k]}^n}{d} \frac{L_{[k]}^{n-m}(f_{[k]}^*H_{[k]})^{m}}{L_{[k]}^{n-m+1}(f_{[k]}^*H_{[k]})^{m-1}} = k^{2\dim A-2} \frac{L^n}{d} \frac{L^{n-m}(f^*H)^m}{L^{n-m+1}(f^*H)^{m-1}}.
	$$
	Apply Theorem \ref{explicitestimate} for $(X_{[k]}, H_{[k]}, L_{[k]})$. It follows from the above equality that 
	\begin{eqnarray*}
		h^0(X_{[k]}, \CO_{X_{[k]}}(L_{[k]})) & \le & \left(\frac{1}{2n!} + \frac{n-\varepsilon}{n!d}\right) L_{[k]}^n + \\ 
		& & k^{2\dim A-2} \frac{2^{n-3}(d+2) L^n}{d} \sum_{m=1}^{n-1} \frac{L^{n-m}(f^*H)^m}{L^{n-m+1}(f^*H)^{m-1}}
	\end{eqnarray*}
	which implies that
	$$
	\frac{h^0(X_{[k]}, \CO_{X_{[k]}}(L_{[k]}))}{k^{2 \dim A}} \le \left(\frac{1}{2n!} + \frac{n-\varepsilon}{n!d}\right) L^n + o(1).
	$$
	Take the limit as $k \to \infty$. Then the proof in this case is completed.
	
	Suppose that $L - f^*H$ is not pseudo-effective. Fix an integer $k>0$ and denote $L'_{[k]} = L_{[k]} + f_{[k]}^*H_{[k]}$. Then $L'_{[k]} - f_{[k]}^*H_{[k]}$ is effective. Using the above argument verbatim, we deduce that
	$$
	h^0_{a_{[k]}}(X_{[k]}, \CO_{X_{[k]}}(L_{[k]} + f_{[k]}^*H_{[k]})) \le \left(\frac{1}{2n!} + \frac{n-\varepsilon}{n!d}\right) (L_{[k]} + f_{[k]}^*H_{[k]})^n.
	$$
	Notice that by Proposition \ref{h0acomparison},
	$$
	h^0_a(X, \CO_X(L)) = \frac{h^0_{a_{[k]}}(X_{[k]}, \CO_{X_{[k]}} (L_{[k]}))}{k^{2 \dim A}} \le \frac{h^0_{a_{[k]}}(X_{[k]}, \CO_{X_{[k]}} (L_{[k]} + f_{[k]}^*H_{[k]}))}{k^{2 \dim A}}.
	$$
	Combine the above results, and it follows from (\ref{numericalequivalence}) that
	$$
	h^0_a(X, \CO_X(L)) \le \left(\frac{1}{2n!} + \frac{n-\varepsilon}{n!d}\right) \frac{(L_{[k]} + f_{[k]}^*H_{[k]})^n}{k^{2 \dim A}} = \left(\frac{1}{2n!} + \frac{n-\varepsilon}{n!d}\right) \left(L + \frac{f^*H}{k^2}\right)^n.
	$$
	Thus the result follows after taking $k \to \infty$.
\end{proof}

\subsection{Volume inequality}
Let $X$ be an $n$-dimensional minimal variety of general type. Assume that the Albanese map $\alb_X: X \to A$ induces a fibration $f: X \to Y$ of curves of genus $g \ge 2$ with $Y$ normal of dimension $n-1$. 

\begin{theorem} [Theorem \ref{main2}] \label{volume}
	Under the above assumption, we have
	$$
	K_X^n \ge 2 n! \left(\frac{g-1}{g+n-2}\right)  \chi(X, \omega_X).
	$$
\end{theorem}

\begin{proof}
	Notice that $Y$ is of maximal Albanese dimension.
	Since $K_X$ is nef and $K_X|_C = K_C$, by Theorem \ref{clifford} (taking $\varepsilon = 1$ as $\deg K_C$ is always even), we obtain
	$$
	K_X^n \ge 2 n! \left(\frac{g-1}{g+n-2}\right)  h^0_{\alb_X}(X, \omega_X).
	$$
	
	On the other hand, it is known that $\omega_X$ has the generic vanishing property. More precisely, in this case, by the generic vanishing theorem  \cite{Green_Lazarsfeld_GenericVanishing} of Green-Lazarsfeld for varieties with rational singularities (see also \cite{Hacon_Kovacs_Generic_Vanishing}), for any general member $\alpha \in \Pic^0(X) = \Pic^0(A)$, 
	$$
	H^i(X, \omega_X \otimes \alpha) = 0, \quad \forall i \ge 2.
	$$
	Thus we deduce that for the above $\alpha$,
	\begin{eqnarray*}
		\chi(X, \omega_X) & = & \chi(X, \omega_X \otimes \alpha) = h^0(X, \omega_X \otimes \alpha) - h^1(X, \omega_X \otimes \alpha) \\
		& \le & h^0(X, \omega_X \otimes \alpha) = h^0_{\alb_X}(X, \omega_X).
	\end{eqnarray*}
	Therefore, the proof is completed.
\end{proof}

\subsection{Slope inequality}
In this subsection, we prove the slope inequality.
\begin{theorem} [Theorem \ref{main3}] 
	Let $f: X \to Y$ be a semi-stable fibration of curves from a normal variety $X$ of dimension $n \ge 2$ to a smooth variety $Y$ of dimension $n-1$ with a general fiber $C$ of genus $g \ge 2$. Suppose that $K_{X/Y}$ is nef. Then the inequality
	$$
	K_{X/Y}^n \ge 2 n! \left(\frac{g-1}{g+n-2}\right) \ch_{n-1}(f_* \omega_{X/Y})
	$$
	holds if there is a finite morphism $Z \to Y$ where $Z$ is either an abelian variety or a smooth toric Fano variety.
	
	Moreover, if $Y$ itself is an abelian variety, then the semi-stability assumption on $f$ can be dropped.
\end{theorem}

\begin{proof}
	Since $f$ is semi-stable and $Y$ is smooth, by \cite[Proposition 2]{Viehweg_Rational_Singularity}, $f$ is flat. Consider the following base change
	$$
	\xymatrix{
		X' \ar[r]^{\tau} \ar[d]_{f'} & X \ar[d]^{f} \\
		Z \ar[r]^{\pi} & Y
	}
	$$
	where $\pi: Z \to Y$ is as in the assumption and $X' = X \times_Y Z$. Apply \cite[Proposition 2]{Viehweg_Rational_Singularity} again. We deduce that $f'$ is also flat and semi-stable, and $X'$ is normal with rational singularities. By \cite[Theorem 3.5.1, 3.6.1]{Conrad_Duality}, we have
	$$
	\tau^* \omega_{X/Y} = \omega_{X'/Z}.
	$$
	We also have $\omega_{X'/Z} = \CO_{X'}(K_{X'/Z})$. Thus $K_{X'/Z}$ is nef. The key fact is that 
	$$
	K_{X'/Z}^{n} = \deg(\pi) K_{X/Y}^n \quad \mbox{and} \quad \ch_{n-1}(f'_* \omega_{X'/Z}) = \deg(\pi) \ch_{n-1}(f_* \omega_{X/Y}).
	$$
	It is now clear that to prove the theorem, it suffices to assume that $Z=Y$ and $f'=f$. In the following, we will assume that $Y$ is an abelian variety or a smooth toric Fano variety.
	
	\textbf{Case 1}. First, suppose that $Y$ is an abelian variety. Then $K_{X/Y} = K_X$ and $\omega_{X/Y} = \omega_X$. By Theorem \ref{clifford}, we obtain
	$$
	K_X^n \ge 2 n! \left(\frac{g-1}{g+n-2}\right)  h^0_f(X, \omega_X) = 2 n! \left(\frac{g-1}{g+n-2}\right)  h^0(Y, f_* \omega_X \otimes \alpha),
	$$
	where $\alpha$ is a general member in $\Pic^0(Y)$. On the other hand, it is know that $f_* \omega_X$ is a generic vanishing sheaf on $Y$ (see \cite{Hacon_GenericVanishing} for instance). In particular,
	$$
	H^i(Y, f_* \omega_X \otimes \alpha) = 0, \quad \forall i \ge 1.
	$$
	Therefore, combine with the Hirzebruch-Riemann-Roch theorem on abelian varieties, and we deduce that
	$$
	h^0(Y, f_* \omega_X \otimes \alpha) = \chi(Y, f_* \omega_X \otimes \alpha) = \chi(Y, f_* \omega_X) = \ch_{n-1}(f_* \omega_X).
	$$
	The proof is completed.
	
	Moreover, if the variety $Y$ in the theorem is already an abelian variety, then we no longer need to make the base change as in the beginning of the proof. Instead, we just apply Theorem \ref{clifford} directly together with the Hirzebruch-Riemann-Roch theorem to deduce the slope inequality as above. Notice that the flatness and the semi-stability are not required in the case.
	
	\textbf{Case 2}. From now on, suppose that $Y$ is a smooth toric Fano variety. Then it is well-known that $Y$ admits a nontrivial \emph{polarized endomorphism} (see \cite{Fakhruddin} for more details). To be more precise, there is a finite morphism $\mu: Y \to Y$ and an ample divisor $H$ on $Y$ such that $\mu^*H \equiv qH$ where $q>1$ is an integer.
	
	For any integer $k>0$, denote by $\mu_k: Y \to Y$ the $k^\mathrm{th}$ iteration of $\mu$. Consider the following commutative diagram:
	$$
	\xymatrix{
		X_{(k)} \ar[r]^{\nu_k} \ar[d]_{f_{(k)}} & X \ar[d]^{f} \\
		Y \ar[r]^{\mu_k} & Y
	}
	$$
	where $X_{(k)} = X \times_{\mu_k} Y$. Similar to the proof at the beginning, for any $k$, $X_{(k)}$ are normal with rational singularities. We still have
	$$
	K_{X_{(k)}/Y}^{n} = \deg(\mu_k) K_{X/Y}^n.
	$$
	
	By taking tensor powers, we may assume that $H$ is very ample. We apply Theorem \ref{explicitestimate} to the triple $(X_{(k)}, H, K_{X_{(k)}/Y})$.  Using a similar limiting argument as in the proof of Theorem \ref{clifford}, we deduce that
	$$
	\frac{h^0(X_{(k)}, \omega_{X_{(k)}/Y})}{\deg(\mu_k)} \le \left(\frac{1}{2n!} + \frac{n-1}{2(g-1)n!}\right) K_{X/Y}^n + o(1)
	$$
	when $k$ is sufficiently large. Therefore, to finish the proof, we only need to prove that 
	$$
	\frac{h^0(X_{(k)}, \omega_{X_{(k)}/Y})}{\deg(\mu_k)} = \frac{\ch_{n-1}(f_{(k)*} \omega_{X_{(k)}/Y})}{\deg(\mu_k)} + o(1).
	$$
	
	In fact, since $\omega_Y^{-1}$ is ample, by the vanishing theorem  of Koll\'ar \cite[Theorem 2.1]{Kollar_Higher} and the fact that $X_{(k)}$ has at worst rational singularities, we deduce that
	$$
	H^i(Y, f_{(k)*} \omega_{X_{(k)}/Y} ) = H^i(Y, f_{(k)*} \omega_{X_{(k)}} \otimes \omega_{Y}^{-1}) = 0, \quad \forall i > 0.
	$$
	Thus it follows from the Hirzebruch-Riemann-Roch theorem that
	$$
	h^0(X_{(k)}, \omega_{X_{(k)}/Y}) = \ch_{n-1}(f_{(k)*} \omega_{X_{(k)}/Y}) + \sum_{j=1}^{n-1} \ch_{n-1-j}(f_{(k)*} \omega_{X_{(k)}/Y}) \frac{T_j}{j!},
	$$
	where $T_j$ is the $j^{\mathrm{th}}$ Todd class of $Y$. Then we only need to prove that
	$$
	\frac{\ch_{n-1-j}(f_{(k)*} \omega_{X_{(k)}/Y}) T_j}{\deg(\mu_k)} \sim o(1)
	$$
	for any $j \ge 1$ when $k$ is sufficiently large. Recall that the Chow ring of smooth toric varieties is generated by divisor classes (see \cite[\S 12.5]{Cox_Toric} for details).  As cycle classes, we may view both $\ch_{n-1-j} (f_{(k)*}\omega_{X_{(k)}/Y})$ and $T_j$ as linear combinations of complete intersections of divisors, thus linear combinations of complete intersections of very ample divisors on $Y$. Notice that we still have
	$$
	\quad \ch_{n-1}(f_{(k)*} \omega_{X_{(k)}/Y}) = \deg(\mu_k) \ch_{n-1}(f_* \omega_{X/Y})
	$$
	here. To prove the above estimate, it is enough to prove that for any $n-1$ very ample divisors $A_1$, \ldots, $A_{n-1}$ on $Y$, 
	$$
	\frac{A_1 \cdots A_j (\mu^*_k A_{j+1}) \cdots (\mu^*_kA_{n-1})}{\deg(\mu_k)} \sim o(1).
	$$
	This is easy to check. By taking tensor powers, we may assume that $H-A_i$ is effective for $1 \le i \le j$. Then 
	\begin{eqnarray*}
		A_1 \cdots A_j (\mu^*_k A_{j+1}) \cdots (\mu^*_kA_{n-1}) & \le & H^j (\mu^*_k A_{j+1}) \cdots (\mu^*_kA_{n-1}) \\
		& = & \frac{1}{q^{kj}} (\mu_k^*H)^j (\mu^*_k A_{j+1}) \cdots (\mu^*_kA_{n-1}) \\
		& = & \frac{\deg(\mu_k)}{q^{kj}} H^jA_{j+1} \cdots A_{n-1}.
	\end{eqnarray*}
    Thus the estimate is proved and then the whole proof is completed.
\end{proof}

\section{Examples and further remarks} \label{examples}

In a communication to us, Hu \cite{Hu2016} provided an interesting construction of a $3$-fold $X$ of general type with $\vol(K_X) = \frac{26}{5} \chi(X, \omega_X)$ whose Albanese map is a fibration of curves of genus $2$ over an abelian surface. In this section (except Section \ref{unbounded}), we provide a general construction completely based on Hu's idea in dimension three. Among others, all examples in this section show the following:
\begin{enumerate}
	\item To guarantee the inequality that $K_X^n \ge n! \chi(X, \omega_X)$ in Theorem \ref{main2}, the genus $g$ of the Albanese fiber cannot be too small. In particular, if $n = 3$, then $g$ cannot be two.
	
	\item In Theorem \ref{main}, the number $\varepsilon$ cannot equal one always, and the nefness assumption on $L$ is crucial.
	
	\item The inequality $L^n \ge 2(n-1)! h^0_f(X, \CO_X(L))$ in Corollary \ref{2n-1!} is sharp.
	
	\item The ratio $K^n_X/\chi(\omega_X)$ is unbounded from above.
\end{enumerate}
Throughout this section, we assume that $g, n \ge 2$ are integers.

\subsection{The construction} \label{constructionhu} Let $A$ be an abelian variety of dimension $n-1$. Fix a very ample divisor $D$ on $A$. Let $Y = \PP(\CO_A \oplus \CO_A(-2D))$ be the $\PP^1$-bundle over $A$, and let $H$ be the effective hyperplane divisor associated to $\CO_Y(1)$. We denote by $p: Y \to A$ the canonical projection.
	
Notice that the linear system $|H+2p^*D|$ is base point free, and we have $\CO_H(H+2 p^*D) = \CO_H$. By Bertini's theorem, we can choose a smooth divisor $H' \in |(2g+1)(H+2p^*D)|$ such that $H'$ and $H$ have no intersection with each other. Let $\pi: X \to Y$ be the double cover branched along $H$ and $H'$. Then the induced map 
$f= p \circ \pi: X \to A$ is a fibration of curves of genus $g$, because a general fiber of $f$ is a double cover of $\PP^1$ ramified along $(2g+2)$ points. From the double cover formula, we have
\begin{eqnarray}
	K_X & \equiv & \pi^*\left( K_Y+(g+1)H + (2g+1)p^*D \right) \label{KX} \\
	& \equiv & \pi^*\left((g-1)H + (2g-1)p^*D \right). \nonumber
\end{eqnarray}
Here $\equiv$ means linear equivalence. Therefore, $K_X$ is ample. Since the branch locus is smooth, we conclude that $X$ is smooth.

This construction will be used throughout this section. 

\subsection{Canonical invariants} In the following, we compute the canonical invariants of $X$.

\begin{lemma} \label{chikn}
	In the above example, we have
	\begin{eqnarray*}
		\chi(X, \omega_X) & = & \frac{D^{n-1}}{(n-1)!} \sum_{i=1}^{g} (2i-1)^{n-1}, \\
		K_X^n & = & \left( (2g-1)^n - 1 \right)D^{n-1}.
	\end{eqnarray*}
\end{lemma}

\begin{proof}
	We first compute $\chi(X, \omega_X)$. By the projection formula, we have
	\begin{eqnarray*}
		\pi_* \omega_X & = & \omega_Y \oplus \left( \omega_Y \otimes \CO_Y \left( (g+1)H + (2g+1) p^*D \right)\right) \\
		& = & \omega_Y \oplus \CO_Y \left((g-1)H + (2g-1)p^*D \right).
	\end{eqnarray*}
	It is easy to see that 
	$$
	\chi(Y, \omega_Y) = (-1)^n \chi(Y, \CO_Y) = (-1)^n \chi(A, \CO_A) = 0.
	$$
	Therefore, it follows that
	\begin{eqnarray*}
		\chi(X, \omega_X) & = & \chi \left(Y, \CO_Y \left((g-1)H + (2g-1)p^*D \right)\right) \\
		& = & \chi \left(A, \mathrm{Sym}^{g-1} \left(\CO_A \oplus \CO_A(-2D) \right) \otimes \CO_A((2g-1)D) \right) \\
		& = & \sum_{i=1}^{g} \chi(A, \CO_A((2i-1)D)) \\
		& = & \frac{D^{n-1}}{(n-1)!} \sum_{i=1}^{g} (2i-1)^{n-1}.
	\end{eqnarray*}
	Hence the equality for $\chi(X, \omega_X)$ is verified. 
	
	To compute $K_X^n$, from \eqref{KX}, we have
	$$
	K_X^n = 2 \left((g-1)H + (2g-1)p^*D\right)^n.
	$$
	Using the fact that $\CO_H(H) = \CO_H(-2p^*D)$ and $H^{n} = (-2D)^{n-1}$, we deduce that
	\begin{eqnarray*}
		\left((g-1)H + (2g-1)p^*D\right)^n & = & \sum_{k=0}^{n} {n \choose k} \left( (g-1) H\right)^{n-k} \left((2g-1)p^*D\right)^k \\
		& = & \sum_{k=0}^{n-1} {n \choose k} \left( (g-1) H\right)^{n-k} \left(- \left(g-\frac{1}{2}\right) H\right)^k \\
		& = & \frac{(-H)^n}{2^n} -  \left(g-\frac{1}{2}\right)^n (-H)^n \\
		& = & \left( \frac{(2g-1)^n - 1}{2} \right) D^{n-1}.
	\end{eqnarray*}
	Hence the proof is completed by combining the above two equalities.
\end{proof}

\begin{prop}
	In the above example, if $n=2$, then
	$$
	K_X^2 = 4 \left(\frac{g-1}{g}\right) \chi(X, \omega_X).
	$$
\end{prop}

\begin{proof}
	Notice that when $n=2$, we have $\chi(X, \omega_X) = g^2 \deg D$ and $K_X^2 = 4g(g-1) \deg D$ by Lemma \ref{chikn}. Hence the result follows.
\end{proof}

In other words, the inequality \eqref{HPX} becomes an equality for the example here. Unfortunately, when $n > 2$, the relation between $K_X^n$ and $\chi(X, \omega_X)$ from this construction will be gradually away from some very plausible expectations.

\begin{prop}
	In the above example, if $n \ge 3$ and $g \le \frac{n+1}{2}$, then 
	$$
	K_X^n < n! \chi(X, \omega_X).
	$$
\end{prop}

\begin{proof}
	By Lemma \ref{chikn}, we only need to check that
	$$
	(2g-1)^n - 1 < n \sum_{i=1}^{g}(2i-1)^{n-1}
	$$
	when $n \ge 2g-1$, and this is clear.
\end{proof}

\begin{remark}
	If we set $g=2$ and $n=3$, then we obtain the  example by Hu in \cite{Hu2016}. From these examples, we also discover that the inequality 
	$$
	K_X^n \ge 2n! \left(\frac{g-1}{g}\right) \chi(X, \omega_X)
	$$
	does not hold under the setting of Theorem \ref{main2} for any $n > 2$.
\end{remark}

\subsection{Remark on the number $\varepsilon$ in Theorem \ref{main}}  \label{epsilon}
Go back to the construction in Section \ref{constructionhu}. As $H$ is contained in the branch locus, we see that $\pi^*H = 2M$ where $M$ is a section of $f: X \to A$. Moreover, as 
$$
2f^*D + 2M = \pi^*(H + 2p^*D),
$$
we deduce that the divisor $f^*D + M$ is nef, which we will denote by $L$.

\begin{prop} \label{surfaceepsilon}
	With the above notation, and let $n = 2$. Then 
	$$
	L^2 = \deg D
	$$
	For any integer $0 < d \le 2g-2$, we have
	$$
	h^0_f(X, \CO_X(dL)) = \left\{ 
	\begin{array}{ll} 
	{\displaystyle \frac{(d+1)^2}{4} \deg D}, & \text{for d odd} \\
	{\displaystyle \frac{d(d+2)}{4} \deg D}, & \text{for d even}.
	\end{array} 
	\right.
	$$
	In particular, when $d$ is odd, 
	$$
	h^0_f(X, \CO_X(dL)) > \left(\frac{1}{4} + \frac{1}{2d}\right) (dL)^2.
	$$
\end{prop}

\begin{proof}
	The first formula is straightforward, because we simply have
	$$
	L^2 = (f^*D+M)^2 = \frac{1}{2} (H+2p^*D)^2 = \deg D.
	$$
	
	Let $C$ be a general fiber of $f$. Then $C$ is hyperelliptic and $2M|_C$ is a $g_2^1$ on $C$. Suppose that $d$ is odd. Then we deduce that $M$ has to be contained in the fixed part $|\CO_X(dL) \otimes f^* \alpha|$ for any $\alpha \in \Pic^0(A)$. By the projection formula, if we let $d-1 = 2e$ and $\alpha$ be general, then
	\begin{eqnarray*}
		h^0_f(X, \CO_X(dL)) & = & h^0_f(X, \CO_X(df^*D + (d-1)M)) \\
		& = & h^0_p(X, \CO_Y(dp^*D + eH)) \\
		& = & \sum_{i=0}^{e} h^0(A, \CO_A((d-2i)D) \otimes \alpha) \\
		& = & \sum_{i=0}^{e} (d-2i) \deg D = \frac{(d+1)^2}{4} \deg D.
	\end{eqnarray*}
	If $d$ is even, similar to the above calculation, we have
	$$
	h^0_f(X, \CO_X(dL)) = \sum_{i=0}^{d/2} h^0(A, \CO_A((d-2i)D) \otimes \alpha) = \frac{d(d+2)}{4} \deg D.
	$$
	The proof is completed.
\end{proof}

\begin{remark}
	Proposition \ref{surfaceepsilon} suggests that in Theorem \ref{main}, we cannot take $\varepsilon = 1$ as long as $d = \deg (L|_C)$ is odd. It also suggests that the inequality in Theorem \ref{main} for surfaces is sharp for any even $d$.
\end{remark}

It is not difficult to generalize the above calculation to any dimension.

\begin{prop} \label{nepsilon}
	With the above notation, and let $n > 2$. Then
	$$
	L^n = D^{n-1}.
	$$
	For any integer $0 < d \le 2g-2$, we have
	$$
	h^0_f(X, \CO_X(dL)) = \left\{ 
	\begin{array}{ll} 
	{\displaystyle \frac{D^{n-1}}{(n-1)!} \sum_{i=0}^{(d-1)/2} (d-2i)^{n-1}}, & \text{for d odd} \\
	{\displaystyle \frac{D^{n-1}}{(n-1)!} \sum_{i=0}^{d/2} (d-2i)^{n-1}}, & \text{for d even}.
	\end{array} 
	\right.
	$$
\end{prop}

\begin{proof}
	Let $\alpha$ be a general element in $\Pic^0(A)$. Similar to the proof in Proposition \ref{surfaceepsilon}, when $d$ is odd, we have 
	\begin{eqnarray*}
		h^0_f(X, \CO_X(dL)) & = & \sum_{i=0}^{(d-1)/2} h^0(A, \CO_A((d-2i)D) \otimes \alpha) \\
		& = & \frac{D^{n-1}}{(n-1)!} \sum_{i=0}^{(d-1)/2} (d-2i)^{n-1}.
	\end{eqnarray*}
	In the last step of the above, we have intrinsically used the Kodaira vanishing theorem. A similar calculation will give the formula for $d$ even.
	
	On the other hand, since
	$(H + 2p^*D)^n = 2^{n-1}D^{n-1}$, we have
	$$
	L^n = 2 \left(\frac{H + 2p^*D}{2} \right)^{n} = D^{n-1}.
	$$
	Hence the proof is completed.
\end{proof}

\begin{remark}
	In Proposition \ref{nepsilon}, when $d=1$, we have
	$$
	h^0_f(X, \CO_X(L)) = \left(\frac{1}{2n!} + \frac{n - \frac{1}{2}}{n!}\right) L^n.
	$$
	This means that at least for this particular case, we have to take $\varepsilon = \frac{1}{2}$ in Theorem \ref{main}. In the case when $d=2$, we get
	$$
	(2L)^n = 2(n-1)! h^0_f(X, \CO_X(2L)).
	$$
	This implies that the inequality in Corollary \ref{2n-1!} is sharp.
\end{remark}

\begin{remark} \label{abigness}
	From Proposition \ref{nepsilon}, it is easy to see that
	$$
	h^0_f(X, \CO_X(2L)) = h^0_f(X, \CO_X(2f^*D)).
	$$
	In other words, it implies that for any general (in fact, any nonzero) $\alpha \in \Pic^0(A)$, we always have the following decomposition of linear systems
	$$
	|\CO_X(2L) \otimes f^* \alpha| = |\CO_X(2f^*D) \otimes f^* \alpha| + 2M 
	$$
	on $X$, where $2M$ is the fixed part of $|\CO_X(2L) \otimes f^* \alpha|$. In particular, the moving part of $|\CO_X(2L) \otimes f^* \alpha|$ is not $f$-big. Notice that in this case, the linear system $|2L|$ itself is even base point free!
\end{remark}

\subsection{Remark on the nefness assumption in Theorem \ref{main}} \label{nefnessremark}

We keep using the setting in Section \ref{epsilon}. That is, let $\pi^*H = 2M$ where $M$ is a section of $f: X \to A$. Now we consider another divisor
$$
L = f^*D + dM
$$
where $2 \le d \le 2g-2$. From \eqref{KX}, we deduce that $K_X - L$ is (linear equivalent to) an effective divisor. Moreover, $L$ is not nef, because $2L = \pi^*(2p^*D + dH)$ and $2p^*D + dH$ is not nef on $Y$. Also, it is easy to see that for any integer $m > 0$,
$$
|2mL| = |2m(f^*D + M)| + 2(d-1)mM,
$$
where $2(d-1)mM$ is the fixed part of $|2mL|$. This implies that 
$$
\vol(L) = \vol(f^*D + M) = D^{n-1} = (n-1)! h^0_f(X, \CO_X(L))
$$
by Proposition \ref{nepsilon}. 

We may consider another divisor
$$
L' = 2f^*D + dM.
$$
If $3 \le d \le 2g-2$ (a priori, $g > 2$), then $K_X-L'$ is also effective but $L'$ is not nef. Similar to the above example, we deduce that
$$
\vol(L') = \vol(2f^*D + 2M) = 2^nD^{n-1} = 2(n-1)! h^0_f(X, \CO_X(L')).
$$

In both examples, the divisor $L$ and $L'$ on $X$ violate the inequality in Theorem \ref{main} due to their non-nefness.

\subsection{Unboundedness of $K^n_X/\chi(\omega_X)$ from above} \label{unbounded}

In this subsection, let $Y = Z \times \mathbb{P}^1$ be the trivial $\mathbb{P}^1$-bundle over $Z$. Instead of an abelian variety, here we assume that $Z$ is a normal and minimal variety of general type and of maximal Albanese dimension with $\dim Z = n-1$. Denote by $p_1$ and $p_2$ the two natural projections from $Y$ to $Z$ and $\mathbb{P}^1$, respectively.

Pick any nef and big divisor $L_1$ on $Z$ with $|L_1|$ base point free, and also pick any effective divisor $L_2$ on $\mathbb{P}^1$ of degree $g+1$. Let $\pi: X \to Y$ be the double cover branched along an divisor $H$ on $Y$, where $H \in |2p_1^*L_1 + 2p_2^* L_2|$. Similar to the previous example, it is easy to show that $X$ is of general type.

We claim that $p_2 \circ \pi: X \to Z$ is just the fibration of genus $g$ curves induced by the Albanese map of $X$. To show this, we only need to prove that
$$
h^1(X, \CO_X) = h^1(Z, \CO_Z).
$$
This is straightforward, because by the projection formula, we have
$$
h^1(X, \CO_X) = h^1(Y, \CO_Y) + h^1(Y, \CO_Y(-p_1^*L_1 - p_2^*L_2)).
$$
Since $h^1(Y, \CO_Y) = h^1(Z, \CO_Z)$ and $h^1(Y, \CO_Y(-p_1^*L_1 - p_2^*L_2)) = 0$ by the Kawamata-Viehweg vanishing theorem, our claim holds.

Similar to the previous calculation, it is easy to show that
$$
K^n_X = 2 (K_Y+p_1^*L_1 + p_2^*L_2)^n = 2n(g-1)(K_Z+L_1)^{n-1}.
$$
In the meantime, we also have
\begin{eqnarray*}
	\chi(X, \omega_X) & = & \chi(X, \omega_Y) + \chi(Y, \omega_Y \otimes \CO_Y(p_1^*L_1 + p_2^*L_2)) \\
	& = & - \chi(Z, \omega_Z) + g \chi(Z, \omega_Z \otimes \CO_Z(L_1) ) \\
	& \le & g \chi(Z, \omega_Z \otimes \CO_Z(L_1) ) .
\end{eqnarray*}
Here we use the fact that $\chi(Z, \omega_Z) \ge 0$ (see \cite{Ein_Lazarsfeld} for example). However, by a very recent result of Barja-Pardini-Stoppino \cite[Example 8.4 and 8.5]{Barja_Pardini_Stoppino}, when $n-1 \ge 3$, the ratio $(K_Z+L_1)^{n-1} / \chi(Z, \omega_Z \otimes \CO_Z(L_1) )$ can be arbitrarily large.

\section{Proof of Corollary \ref{2n-1!}} \label{proofcoro}

In this section, we will give a proof of Corollary \ref{2n-1!}. Resume the notation in Theorem \ref{main}. Suppose from now that
$$
L^n = 2(n-1)! h^0_f(X, \CO_X(L)) > 0.
$$ 
Then $d = 2$ by Theorem \ref{main}. In the following, we only need to prove that in this case, $Y$ is birational to an abelian variety. 

\subsection{Set up} Similar to the proof of Theorem \ref{clifford}, we consider the following commutative diagram
$$
\xymatrix{
	X_{[k]} \ar[r]^{\phi_{[k]}} \ar[d]_{f_{[k]}} 
	\ar@/_2.5pc/[dd]_{a_{[k]}} & X \ar[d]^f  \ar@/^2.5pc/[dd]^a \\
	Y_{[k]} \ar[r]^{\psi_{[k]}} \ar[d]_{b_{[k]}} & Y \ar[d]^b \\
	A \ar[r]^{[k]} & A}
$$
where every notation is exactly the same as that in the proof of Theorem \ref{clifford}. Write $L_{[k]} = \phi^*_{[k]} L$. By Corollary \ref{h0acomparison} and Lemma \ref{h0fh0a}, we have
\begin{eqnarray*}
	L_{[k]}^n = k^{2 \dim A}L^n & = & 2k^{2 \dim A} (n-1)! h^0_f(X, \CO_X(L)) \\
	& = & 2(n-1)! h^0_{a_{[k]}}(X_{[k]}, \CO_{X_{[k]}}(L_{[k]})).
\end{eqnarray*}

Since $L^n > 0$, we may assume that $L_{[k]}$ is effective by replacing $\CO_{X_{[k]}}(L_{[k]})$ by its tensor with $a^*_{[k]} \alpha$, where $\alpha \in \Pic^0(A)$ is general. Notice that we have 
$$
h^0_{a_{[k]}}(X_{[k]}, \CO_{X_{[k]}}(L_{[k]})) = h^0(X_{[k]}, \CO_{X_{[k]}}(L_{[k]})) \sim k^{2 \dim A}.
$$

Let $G \ge 0$ be a sufficiently ample divisor on $A$ and let $H_{[k]} = b^*_{[k]}G$. As the main result in this step, we claim that 
$$
h^0(X_{[k]}, \CO_{X_{[k]}} (L_{[k]} - f^*_{[k]} H_{[k]} )) > 0
$$
when $k$ is sufficiently large.

To prove this, it is enough to prove that 
$$
h^0(Y_{[k]}, \CO_{Y_{[k]}} (H_{[k]})) \sim o(k^{2 \dim A}).
$$
In fact, it is easy to see that the complete linear system $|H_{[k]}|$ defines a generically finite morphism on $Y_{[k]}$, since the sub-linear system $b^*_{[k]}|G|$ has already defined a generically finite morphism. Therefore, we have
$$
h^0(Y_{[k]}, \CO_{Y_{[k]}} (H_{[k]})) - (n-1) \le H^{n-1}_{[k]} =  G^{n-1} \deg (b_{[k]}) =  G^{n-1} \deg (b).
$$
Thus the proof of the claim is completed. 

Consider the rational map
$$
\phi_{L_{[k]}}: X_{[k]} \dashrightarrow \PP^{h^0(X_{[k]}, \CO_{X_{[k]}}(L_{[k]})) - 1}
$$
defined by the complete linear system $|L_{[k]}|$. Then the above claim implies that $a_{[k]}$ must factor through $\phi_{L_{[k]}}$. In particular, 
$$
\dim \phi_{L_{[k]}} (X_{[k]}) \ge n-1.
$$
The proof will be divided to two cases subject to this dimension.

In the following, by abuse of the notation, we denote also by $C$ a general fiber of $f_{[k]}$, because it is isomorphic to a general fiber of $f$.

\subsection{Case I} In this subsection, we assume that $\dim \phi_{L_{[k]}} (X_{[k]}) = n - 1$ for any $k$ sufficiently large. 

Since $a_{[k]}$ factors through $\phi_{L_{[k]}}$, we deduce that via Stein factorization, $\phi_{L_{[k]}}$ and $a_{[k]}$ must factor through the same map with connected fibers, which is nothing but $f_{[k]}$. By \cite[Theorem 3.3]{Barja_Severi}, choosing an integer $k$ sufficiently large and blowing up $X_{[k]}$ and $Y_{[k]}$ if necessary, we can write
$$
L_{[k]} = f^*_{[k]} M + Z
$$
such that for any general $\alpha \in \Pic^0(A)$, the line bundle $\CO_{Y_{[k]}}(M) \otimes b^*_{[k]} \alpha$ is globally generated and
$$
h^0_{a_{[k]}}(X_{[k]}, \CO_{X_{[k]}}(L_{[k]})) = h^0(Y_{[k]}, \CO_{Y_{[k]}}(M) \otimes b^*_{[k]} \alpha) = h^0_{b_{[k]}}(Y_{[k]}, \CO_{Y_{[k]}}(M)).
$$

Since $b_{[k]}$ is generically finite, by \cite[Main Theorem (c)]{Barja_Severi}, we have 
$$
M^{n-1} \ge (n-1)! h^0_{b_{[k]}}(Y_{[k]}, \CO_{Y_{[k]}}(M)).
$$
On the other hand, notice that $\deg(Z|_C) = d = 2$. We deduce that 
\begin{eqnarray*}
	2 (n-1)! h^0_{a_{[k]}}(X_{[k]}, \CO_{X_{[k]}}(L_{[k]})) = L_{[k]}^n \ge L_{[k]}(f^*_{[k]}M)^{n-1} = (f^*_{[k]}M)^{n-1} Z = 2 M^{n-1},
\end{eqnarray*}
i.e.,
$$ 
(n-1)! h^0_{b_{[k]}}(Y_{[k]}, \CO_{Y_{[k]}}(M)) \ge M^{n-1}.
$$
As a result, we obtain that
$$
M^{n-1} = (n-1)! h^0_{b_{[k]}}(Y_{[k]}, \CO_{Y_{[k]}}(M)).
$$
Finally, the claim in the previous subsection actually implies that $M$ is also big, because $H_{[k]}$ is big and $M-H_{[k]}$ is effective. By a very recent result \cite[Theorem 7.1]{Barja_Pardini_Stoppino} of Barja-Pardini-Stoppino, we know that $b_{[k]}$ is birational. So is $b$. Therefore, $Y$ is birational to an abelian variety.

\subsection{Case II} In this subsection, we assume that $\dim \phi_{L_{[k]}} (X_{[k]}) = n$ for some $k$. 

In this case, $\phi_{L_{[k]}}(C)$ must have dimension one. Notice that $\deg(L_{[k]}|_C) = 2$ and $g(C) \ge 2$. We conclude that $C$ is hyperelliptic and $\phi_{L_{[k]}}$ induces a double cover from $C$ to $\PP^1$. In particular, $\deg (\phi_{L_{[k]}}) \ge 2$.

Let $\Sigma \subseteq \PP^{h^0(X_{[k]}, \CO_{X_{[k]}}(L_{[k]})) - 1}$ be the image of $X$ under $\phi_{L_{[k]}}$. By factoring through the Stein factorization, we may assume that $\Sigma$ is normal. Replacing $\Sigma$ by a resolution of singularities and then blowing up $X_{[k]}$ accordingly, we may assume further that $\Sigma$ is smooth and $\phi_{L_{[k]}}$ is a morphism. Moreover, we can write
$$
|L_{[k]}| = |M| + Z,
$$
where the movable part $|M|$ of $|L_{[k]}|$ is base point free and $M = \phi_{L_{[k]}}^*N$ for a divisor $N$ on $\Sigma$. 

Recall that $a_{[k]}$ factors through $\phi_{L_{[k]}}$. We write $c: \Sigma \to A$ such that $a_{[k]} = c \circ \phi_{L_{[k]}}$. Then $c^*: \Pic^0(A) \to \Pic^0(\Sigma)$ is injective. Similar to the previous case, we may also assume that up to tensoring with a general $\alpha \in \Pic^0(a)$, we have
$$
h^0_c(\Sigma, \CO_{\Sigma}(N)) = h^0(X_{[k]}, \CO_{X_{[k]}}(M) \otimes a_{[k]}^* \alpha) = h^0(X_{[k]}, \CO_{X_{[k]}}(L_{[k]})).
$$
Thus it follows that 
$$
2 (n-1)! h^0_c(\Sigma, \CO_{\Sigma}(N)) = 2(n-1)! h^0(X_{[k]}, \CO_{X_{[k]}}(L_{[k]}))= L_{[k]}^n \ge M^n \ge 2N^n,
$$
i.e.,
$$
N^n \le (n-1)! h^0_c(\Sigma, \CO_{\Sigma}(N)).
$$

By Bertini's theorem, we choose a smooth subvariety $W \in |N|$ of dimension $n-1$ on $\Sigma$. Then $c^*:\Pic^0(A) \to \Pic^0(W)$ is also injective. In fact, we have proved that $\Sigma$ is birational to a $\PP^1$-bundle and $W$ has degree one when restricted on a general fiber of this $\PP^1$-bundle. With all these notation, we consider a similar commutative diagram as follows:
$$
\xymatrix{
	W_m \ar[r]^{p_m} \ar@{^{(}->}[d] 
	\ar@/_2.5pc/[dd]_{\iota_m} & W \ar@{^{(}->}[d] \ar@/^2.5pc/[dd]^{\iota} \\
	\Sigma_m \ar[r]^{q_m} \ar[d]_{c_m} & \Sigma \ar[d]^c \\
	A \ar[r]^{[m]} & A}
$$
where $[m]: A \to A$ is the multiplication map by $m$, $\Sigma_m = \Sigma \times_{[m]} A$ and $W_m = W \times_{[m]} A$. Let $N_m = q_m^* N$. Notice that $W_m \in |N_m|$ is also smooth and irreducible. By \cite[Main Theorem (c)]{Barja_Severi} again, we deduce that for any $m$,
\begin{eqnarray*}
	N_m^n = (N_m|_{W_m})^{n-1} & \ge & (n-1)! h^0_{\iota_m}(W_m, \CO_{W_m}(N_m)) \\
	& \ge & (n-1)! \left( h^0_{c_m}(\Sigma_m, \CO_{\Sigma_m}(N_m)) - 1 \right).
\end{eqnarray*}
Apply Proposition \ref{h0acomparison} as well as the previous result. Then the above inequality is equivalent to
\begin{eqnarray*}
	(n-1)! h^0_c(\Sigma, \CO_{\Sigma}(N)) \ge  (N|_W)^n & \ge & (n-1)! h^0_\iota(W, \CO_{W}(N)) \\ 
	& \ge & (n-1)! \left( h^0_c(\Sigma, \CO_{\Sigma}(N)) - \frac{1}{m^{2 \dim A}} \right).
\end{eqnarray*}
Take the limit as $m \to \infty$. We deduce that all inequalities above should be equalities. In particular,
$$
N^n = (N_W)^{n-1} = (n-1)! h^0_\iota(W, \CO_{W}(N)) = (n-1)! h^0_c(\Sigma, \CO_{\Sigma}(N)).
$$
By \cite[Theorem 7.1]{Barja_Pardini_Stoppino} again, we deduce that $\iota$ is birational. Meanwhile, all equalities force $\deg (\phi_{L_{[k]}}) = 2$. This means that $\phi_{L_{[k]}}$ has to separate any two general fibers of $f_{[k]}$, because $\phi_{L_{[k]}}$ has already induced a double cover on each general fiber. As a result, we deduce that not only $a_{[k]}$, but also $f_{[k]}$ should factor through $\phi_{L_{[k]}}$, i.e., we have the following commutative diagram
$$
\xymatrix{
	& & X_{[k]} \ar[d]^{f_{[k]}} \ar[ld]_{\phi_{L_{[k]}}} \ar@/^2.5pc/[dd]^{a_{[k]}}\\
	W \ar@{^{(}->}[r] \ar[rrd]_{\iota} & \Sigma \ar[r] \ar[rd]^{c} & Y_{[k]} \ar[d]^{b_{[k]}} \\
	& & A
}
$$
As $\iota$ is birational, we know that $b_{[k]}$ is also birational, so is $b$. This completes the whole proof of Corollary \ref{2n-1!}.

\subsection{Remark on the case $d=1$} Similar to Corollary \ref{2n-1!}, we also have the following result. 

\begin{coro} \label{n-1!}
	Under the setting in Theorem \ref{main}, we have a sharp inequality
	$$
	L^n \ge (n-1)! h^0_f(X, \CO_X(L)).
	$$
	
	Suppose that the equality holds for $L$ and $h^0_f(X, \CO_X(L)) > 0$ (or equivalently, $L$ is big). Then $d=1$ and $Y$ is birational to an abelian variety of dimension $n-1$.
\end{coro}

The proof is almost the same as Case I in the proof of Corollary \ref{2n-1!}. The inequality itself is directly by Theorem \ref{main}. If the equality holds, then $d=1$ is straightforward, and it implies that $\dim \phi_{L_{[k]}} (X_{[k]}) = n-1$ when $k$ is sufficiently large. Therefore, the proof of Case I applies here almost identically, and the only essential modification we need to make is to set $\deg(Z|_C) = d = 1$ accordingly. We leave the proof to the interested reader.

\section{A question of Reid's type for irregular varieties} \label{question}

We would like to end this paper by raising the following question regarding the geography of irregular varieties of general type and of \emph{non-maximal} Albanese dimension. 

\begin{ques} \label{finalquestion}
	Let $n > 1$ be an integer. For any fixed integer $0 < k < n$, is there a sequence of rational numbers $\{a_{k, m}\}_{m \in \mathbb{N}}$ with the following three properties?
	\begin{itemize}
		\item[(i)] $0 < a_{k, 1} < a_{k, 2} < \cdots$;
		\item[(ii)] $\lim_{m \to \infty} a_{k, m} = 2(k+1)!$;
		\item[(iii)] For every smooth $n$-dimensional variety $X$ of general type and of Albanese dimension $k$, if $\vol(K_X) < a_{k, m} \chi(X, \omega_X)$, then a general Albanese fiber of $X$ has volume at most $m$.
	\end{itemize}
\end{ques}

It is clear that Question \ref{finalquestion} is an analogue of Reid's conjecture in \cite{Reid_Quadrics} for irregular varieties in any dimension. The main difference is that we can add in here the information about the Albanese dimension of $X$ to make the question look a bit more elaborate. 

The following is a list of results known to us towards this question by far:
\begin{enumerate}
	\item The inequality \eqref{HPX} by Horikawa \cite{Horikawa_V}, Persson \cite{Persson_Chern} and Xiao \cite{Xiao_Slope} provides an affirmative answer for $n=2$ (where $k=1$);
	
	\item Theorem \ref{main4} is an affirmative answer for $n=3$ (where $k=1$ or $2$);
	
	\item Theorem \ref{main2} is an affirmative answer when $k=n-1$.
\end{enumerate}
It should be interesting to investigate this question in more cases.


\bibliography{Relative_Clifford_inequality_v11}
\bibliographystyle{amsplain}

\end{document}